\documentclass[a4paper,11pt,english]{smfart}
\usepackage[applemac]{inputenc}
\usepackage[english]{babel}
\usepackage{amsfonts}
\usepackage{amsmath} \usepackage{amsthm}
\usepackage{hyperref} 
\usepackage{latexsym}
\usepackage{array}
\usepackage{amssymb}
\usepackage{enumerate}
\usepackage{smfthm}
\usepackage{graphicx}
\usepackage{color}
\usepackage{stmaryrd}
\newcommand{\ligne}{\vspace{1\baselineskip}}
\newcommand{\ph}{\phantomsection}
\newcommand{\cal}{\mathcal}
\newcommand{\di}{\displaystyle}

\newcommand{\R}{\mathbb  R}
\newcommand{\C}{\mathbb  C}
\newcommand{\N}{\mathbb  N}

\newcommand{\1}{\mathbb  I}

\renewcommand{\P}{\mathbb{P}}

\newcommand{\W}{  \mathcal{W}   }

\newcommand{\M}{  \mathcal{M}  }
\newcommand{\eps}{\varepsilon}
\renewcommand{\epsilon}{\varepsilon}
\newcommand{\e}{  \text{e}   }

\newcommand{\E}{\mathbb{E} }

\renewcommand{\H}{  \mathcal{H}   }
\newcommand{\A}{  \mathcal{A}   }

\newcommand{\dis}{\displaystyle}

\newcommand{\om}{  \omega   }

\renewcommand{\a}{  \alpha   }

\newcommand{\lessim}{  \lesssim  }
\renewcommand{\phi}{  \varphi  }
\renewcommand{\L}{  \mathcal{L}   }

\newcommand{\<}{  \langle   }

\renewcommand{\>}{  \rangle   }

\numberwithin{equation}{section}
\theoremstyle{plain}

\newtheorem*{acknowledgements}{Acknowledgements}
\newtheorem{assumption}{Assumption}

\def\beq{\begin{equation}}   \def\eeq{\end{equation}}
\def\bea{\begin{eqnarray}}  \def\eea{\end{eqnarray}}

\renewcommand{\theequation}{\thesection.\arabic{equation}}
\newcounter{hran} \renewcommand{\thehran}{\thesection.\arabic{hran}}

\def\bmini{\setcounter{hran}{\value{equation}}
    \refstepcounter{hran}\setcounter{equation}{0}
    \renewcommand{\theequation}{\thehran\alph{equation}}\begin{eqnarray}}

\def\bminiG#1{\setcounter{hran}{\value{equation}}
\refstepcounter{hran}\setcounter{equation}{-1}
\renewcommand{\theequation}{\thehran\alph{equation}}
\refstepcounter{equation}\label{#1}\begin{eqnarray}}

\author{ Aur\'elien Poiret}
\address{Laboratoire de Math\'ematiques, UMR 8628 du CNRS.
Universit\'e Paris Sud, 91405 Orsay Cedex, France}
\email{aurelien.poiret@math.u-psud.fr}
\author{ Didier Robert}
\address{Laboratoire de Math\'ematiques J. Leray, UMR  6629 du CNRS, Universit\'e de Nantes, 
2, rue de la Houssini\`ere,
44322 Nantes Cedex 03, France}
\email{didier.robert@univ-nantes.fr}
\author{ Laurent Thomann }
\address{Laboratoire de Math\'ematiques J. Leray, UMR  6629 du CNRS, Universit\'e de Nantes, 
2, rue de la Houssini\`ere,
44322 Nantes Cedex 03, France}
\email{laurent.thomann@univ-nantes.fr}

 \title[Random  weighted Sobolev inequalities on $\R^d$]{Random  weighted Sobolev inequalities on $\R^d$ and application to Hermite functions}

\begin{document}
\frontmatter

 \begin{abstract}
We extend a randomisation method, introduced by Shiffman-Zelditch and developed by Burq-Lebeau on compact manifolds for the Laplace operator, to the case of $\R^d$ with the harmonic oscillator.   We construct measures, thanks to probability laws which satisfy the concentration of measure  property, on the support of which we  prove optimal weighted Sobolev estimates on~$\R^d$. This construction relies on accurate estimates on the spectral function in a non-compact configuration space. As an application, we show that there exists a basis of Hermite functions with good decay properties in $L^{\infty}(\R^{d})$, when $d\geq 2$.
  \end{abstract}
\thanks{D. R.  was partly supported  by the  grant  ``NOSEVOL''   ANR-2011-BS01019 01.\\L.T. was partly supported  by the  grant  ``HANDDY'' ANR-10-JCJC 0109}
\maketitle
\mainmatter

 \section{Introduction and results}
 
  \subsection{Introduction}\label{intro}
During the last  years, several papers have shown that  some basic results concerning P.D.E. and Sobolev spaces  can be strikingly improved using randomization techniques. In particular  Burq-Lebeau developed in  \cite{bule} a randomisation method based on the Laplace operator  on a compact Riemannian manifold, and showed that almost surely, a function enjoys better Sobolev estimates than expected, using ideas of Shiffman-Zelditch \cite{ShZe}. 
This 
approach depends heavily on spectral properties of the operator one considers.   In this paper we  are interested in estimates in Sobolev spaces based on the harmonic oscillator      in~$L^2(\R^d)$
\begin{equation*}  
H=-\Delta+|x|^{2}=\sum_{j=1}^{d}(-\partial^{2}_{j}+x^{2}_{j}).
\end{equation*}
 \indent  We get optimal stochastic weighted Sobolev estimates on $\R^d$ using the Burq-Lebeau method. Indeed we show that   there is a unified  setting  for these results, including the case of compact manifolds. We also make the following extension: In  \cite{bule}, the construction of the measures relied on Gaussian random variables, while in our work we consider general random variable which satisfy concentration of measure estimates (including discrete random variables, see Section \ref{Sect2}).  However, we obtain the optimal estimates only in the case of the Gaussians. 
 
We will see that the extension from a compact manifold to an operator on $\R^{d}$ with discrete spectrum is not trivial because of the complex behaviour of the spectral function on a non-compact configuration space.

In our forthcoming paper \cite{PRT2}, we will give some applications to the well-posedness of  nonlinear Schr\"odinger equations with Sobolev regularity below the optimal deterministic index.\ligne
 
Most of the results stated here  can be extended to  more general Schr\"odinger Hamiltonians $-\triangle + V(x)$
with  confining potentials $V$.  This will be detailed in \cite{PRT3}. \ligne

 Let $d\geq 2$. We want to define probability measures on finite dimensional subspaces $\cal{E}_h \subset L^2(\R^d)$, based on spectral   projections with respect to $H$. We denote by $\{\phi_{j},\;j\geq 1\}$ an orthonormal basis of~$L^{2}(\R^{d})$ of eigenvectors  of~$H$ (the Hermite functions), and we denote by $\{\lambda_j,\;j\geq 1\}$  the non decreasing  sequence of eigenvalues   (each is repeated according to its multiplicity): $H \phi_{j} =\lambda_j \phi_{j}$.\ligne
  
  For $h>0$, we define the interval   $I_h = [\frac{a_h}{h}, \frac{b_h}{h}[$  and we assume that $a_{h}$ and $b_{h}$ satisfy, for some $a,b,D>0, \delta\in[0, 1]$, 
\begin{equation}\label{spect1}
 \lim_{h\rightarrow 0}a_h =a,  \;\;\; \lim_{h\rightarrow 0}b_h = b,\quad 0<a \leq b\quad \text{and}\quad  b_h -a_h \geq Dh^\delta,
\end{equation}
with  any $D>0$ if $\delta<1$ and  $D\geq2$ in the case $\delta=1$. This condition ensures that $N_{h}$,  the number (with multiplicities) of eigenvalues of $H$ in $I_h$ tends to infinity when $h\to 0$.
Indeed,  we can check that $N_{h}\sim ch^{-d}(b_{h}-a_{h})$, in particular $\dis \lim_{h\rightarrow 0}N_h = +\infty$, since $d\geq 2$.  
 In the sequel, we write $\Lambda_h = \{j\geq 1, \;\lambda_j\in I_h\}$ and ${\cal E}_h = {\rm span}\{\phi_{j},\;j\in \Lambda_h\}$, so that  $N_h  = \#\Lambda_h={\rm dim}\,{\cal E}_h$. Finally, we denote by ${\bf S}_h=\big\{u\in \cal{E}_h\,:\; \Vert u\Vert_{L^2(\R^d)} =1\big\}$ the unit sphere  of~$\cal{E}_h$.\ligne
  
 In the sequel, we will consider sequences $(\gamma_{n})_{n\in \N}$  so that  there exists $K_{0} >0$ 
 \beq\label{condi1}
 \vert\gamma_n\vert^2 \leq \frac{K_0}{N_h}\sum_{j\in\Lambda_h}\vert\gamma_j\vert^2, \quad \forall n\in\Lambda_h,\;\;\forall h\in ]0, 1].
\eeq
This condition  means that on each level of energy $\lambda_n$, $n\in\Lambda_h$, one  coefficient $|\gamma_{k}|$ cannot be much larger  than the others. Sometimes, in order to prove lower bound estimates, we will need the stronger condition ($K_{1}>0$)
\beq\label{condi0}
 \frac{K_1}{N_h}\sum_{j\in\Lambda_h}\vert\gamma_j\vert^2 \leq \vert\gamma_n\vert^2 \leq \frac{K_0}{N_h}\sum_{j\in\Lambda_h}\vert\gamma_j\vert^2, \quad \forall n\in\Lambda_h,\;\;\forall h\in ]0, 1].
\eeq
  This    so-called ``squeezing'' condition means that on each level of energy $\lambda_n$, $n\in\Lambda_h$,  the coefficients~$|\gamma_{k}|$ have almost the same size. For instance \eqref{condi1} or \eqref{condi0} hold if there exists $(d_{h})_{h\in]0,1]}$ so that $\gamma_{n}=d_{h}$ for all $n\in \Lambda_h$.\ligne
 
Consider a   probability space
$(\Omega, {\cal F}, {\P})$  and  let $\{X_n,\;n\geq1\}$ be  independent standard complex Gaussians $\mathcal{N}_{\C}(0,1)$. In fact, in our work we will consider more general probability laws, which satisfy concentration of measure estimates (see Assumption \ref{Assum1}), but for sake of clarity, we first state the results in this particular case. If $(\gamma_{n})_{n\in \N}$ satisfies \eqref{condi1}, we  define the random vector  in ${\cal E}_h$ 
\begin{equation*}
v_{\gamma}(\omega) : = v_{\gamma,h} (\omega)= \sum_{j\in\Lambda_{h}}\gamma_jX_j(\omega)\phi_{j}.
\end{equation*}
 We  define a probability measure ${\bf P}_{\gamma,h}$ on ${\bf S}_h$ by: for all measurable and bounded function $f:{\bf S}_h\longrightarrow \R$
 \begin{equation*}
  \int_{{\bf S}_h}f(u)\text{d}{\bf P}_{\gamma,h}(u) =  \int_\Omega f\left(\frac{v_\gamma(\omega)}{\Vert v_\gamma(\omega)\Vert_{L^{2}(\R^{d})}}\right)\text{d}\P(\omega).
 \end{equation*}
We can check that in the isotropic case ($\gamma_j = \frac{1}{\sqrt {N_{h}}}$ for all $j\in\Lambda_{h}$), 
 ${\bf P}_{\gamma,h}$  is the uniform probability on~${\bf S}_h$ (see Appendix \ref{AppC}). \ligne
 
  Finally, let us recall the definition of harmonic Sobolev spaces for $s\geq 0$, $p\geq 1$. 
 \begin{equation}\label{sobharm}
         \W^{s, p}= \W^{s, p}(\R^d) = \big\{ u\in L^p(\R^d),\; {H}^{s/2}u\in L^p(\R^d)\big\},     
       \end{equation}
       \begin{equation*}
           \H^{s}=   {\cal H}^{s}(\R^d) = \W^{s, 2}.
       \end{equation*}
             The natural norms are denoted by $\Vert u\Vert_{\W^{s,p}}$ and up to equivalence of norms we have (see \cite[Lemma~2.4]{YajimaZhang2}) 
             for $1<p<+\infty$
        $$
      \Vert u\Vert_{\W^{s,p}} = \Vert  H^{s/2}u\Vert_{L^{p}} \equiv \Vert (-\Delta)^{s/2} u\Vert_{L^{p}} + 
       \Vert\<x\>^{s}u\Vert_{L^{p}}.
       $$
 
  \subsection{Main results of the paper} 
  
  \subsubsection{\bf Estimates  for frequency localised functions} 
Our first result gives properties of the elements on the support of ${\bf P}_{\gamma,h}$, which are high frequency localised functions. Namely 
 
  \begin{theo} \ph\label{thm1}
Let $d\geq 2$. Assume that $0\leq \delta<2/3$ in \eqref{spect1} and that condition \eqref{condi0} holds. Then there exist $0<C_0 < C_1$,    $c_1>0$ and $h_0>0$ such that for all $h\in]0, h_0]$.
\begin{equation*} 
{\bf P}_{\gamma,h}\left[u\in{\bf S}_h : C_0\vert\log h\vert^{1/2}  \leq \Vert u\Vert_{\W^{d/2,\infty}(\R^{d})} 
\leq C_1\vert\log h\vert^{1/2}\,  \right] \geq 1- h^{c_1}.
\end{equation*}
Moreover the estimate  from above is satisfied for  any $\delta \geq 1$ with $D$ large enough.
\end{theo}

It is clear that under condition \eqref{condi0}, there exist $0<C_2 < C_3$, so that for all $u\in {\bf S}_h$, and $s\geq 0$
\begin{equation*}
C_2h^{-s/2}  \leq \Vert u\Vert_{\H^{s}(\R^{d})}\leq C_3 h^{-s/2},
\end{equation*}
since all elements of ${\bf S}_h$ oscillate with frequency $h^{-1/2}$. Thus Theorem \ref{thm1} shows a gain of $d/2$ derivatives in $L^{\infty}$, and this induces a gain of $d$ derivatives compared to the usual deterministic Sobolev embeddings. This  can be compared with the results of \cite{bule} where the authors obtain a gain of ~$d/2$ derivatives on compact manifolds: this comes from different behaviours of the spectral function, see Section \ref{Sect3}.  
Notice that the bounds in Theorem \ref{thm1} (and in the results of \cite{bule} as well) do not depend on the length of the interval of the frequency localisation $I_{h}$ (see \eqref{spect1}), but only on the size of the frequencies. This is a consequence of the randomisation, and from the bound \eqref{dsp}.
\ligne

We will see in Theorem \ref{Thm41} that the upper bound in Theorem \ref{thm1} holds for any $0\leq \delta\leq1$ and for more general random variables $X$ which satisfy the concentration of measure property. However, to prove the lower bound (see Corollary \ref{2ESTx}), we have to restrict to the case of Gaussians: in the general case, under Assumption \ref{Assum1}, we do not reach the factor $\vert \ln h \vert^{1/2}$. Following the approach of \cite{ShZe, bule}, we first prove estimates of $\Vert u\Vert_{\W^{d/2,\infty}(\R^{d})}$ with large $r$ and uniform constants (see Theorem \ref{thm412}), and which are essentially optimal for general random variables (see Theorem \ref{thm413}).

The condition $\delta<2/3$ is needed to prove the lower bound, thanks to a reasonable  functional calculus based on the harmonic oscillator (see Appendix \ref{AppB}). \ligne

Finally we point out that in a very recent paper \cite{feze}, Feng and Zelditch  prove similar estimates for the mean and median for the $L^\infty$-norm of random holomorphic fields.\ligne

\subsubsection{\bf Global  Sobolev estimates} Using a dyadic Littlewood-Paley decomposition, we now give general estimates in Sobolev spaces; we refer to  Subsection \ref{UB}  for more details.   For $s\in\R$ and~$p, q\in[1,+\infty]$, we define the harmonic Besov space by 
         \begin{equation}\label{Besov}
           \di{{\cal B}_{p,q}^s(\R^d) = \Big\{u=\sum_{n\geq 0}u_n: \;\; 
         \sum_{n\geq 0}2^{nqs/2}\Vert u_n\Vert_{L^p(\R^d)}^q <+\infty\,\Big\}},
         \end{equation}
   where the $u_n$ have frequencies of size $\sim 2^n$. The  space ${\cal B}_{p,q}^s(\R^d)$    is a Banach space  with the norm in~$\ell^q(\N)$ of $\{2^{ns/2}\Vert u_n\Vert_{L^p(\R^d)}\}_{n\geq 0}$.
   
         We assume that  $\gamma$ satisfies (\ref{condi1})  and 
         $$
         \sum_{n\geq 0}\vert\gamma\vert_{\Lambda_n} < +\infty \quad \text{where}\quad \vert\gamma\vert^{2}_{\Lambda_n}:= \sum_{k : \,\lambda_{k} \in [2^{k},2^{k+1}[}\vert\gamma_{k}\vert^{2} .
         $$
         Then we set 
         $$
  v_\gamma(\omega) = \sum_{j=0}^{+\infty} \gamma_jX_j(\omega)\varphi_j,
$$
so that almost surely $v_\gamma  \in  {\cal B}_{2, 1}^0(\R^d)$   and its  probability law  defines a measure 
$\mu_\gamma$ in $ {\cal B}_{2, 1}^0(\R^d)$. Notice that we have
         $$
         {\cal H}^s(\R^d) \subset {\cal B}_{2, 1}^0(\R^d) \subset L^2(\R^d),\quad \forall s>0.
         $$
         We have the following result
        \begin{theo}\ph\label{Stric}
        For every $(s,r)\in\R^2$ such that $r\geq 2$ and $s=d(\frac{1}{2} -\frac{1}{r})$
        there exists $c_0>0$ such that for all $K>0$ we have
        \beq\label{lp0}
        \mu_{\gamma}\Big[\,u\in {\cal B}_{2, 1}^0(\R^d):\; \Vert u\Vert_{{\cal W}^{s,r}(\R^d)} \geq  K\Vert u\Vert_{{\cal B}_{2, 1}^0(\R^d)} \,\Big]\leq {\rm e}^{-c_0K^2}.
        \eeq
        In particular $\mu_\gamma$-almost all functions in ${\cal B}_{2, 1}^0(\R^d)$ are in 
        ${\cal W}^{s,r}(\R^d)$.
        \end{theo}

If      $\gamma$ satisfies (\ref{condi1})  and the (weaker) condition
         $
      \dis   \sum_{n\geq 0}\vert\gamma\vert^2_{\Lambda_n} < +\infty,
         $
then $\mu_\gamma$ defines a probability  measure on $L^2(\R^d)$ and we can prove the estimate
   \begin{equation}
     \mu_{\gamma}\Big[\,u\in L^2(\R^d):\; \Vert u\Vert_{{\cal W}^{s,r}(\R^d)} \geq  K\Vert u\Vert_{L^2(\R^d)} \,\Big]\leq {\rm e}^{-c_0K^2},
   \end{equation}
with  $s= d(\frac{1}{2} -\frac{1}{r})$ when $r<+\infty$ and $s< d/2$ in the case $r=+\infty$.   From this result it is easy to deduce space-time estimates (Strichartz) for the linear flow $\e^{-itH}u$, which can be used to study the nonlinear problem. This will be pursued in  \cite{PRT2}.

  \subsubsection{\bf An application to Hermite functions} 
Similarly to \cite{bule}, the previous results give some information on Hilbertian bases. We prove that there exists a basis of Hermite functions with good decay properties. 
\begin{theo}\label{theoB}
Let $d\geq 2$. Then there exists a Hilbertian basis of $L^{2}(\R^{d})$ of eigenfunctions of the harmonic oscillator $H$ denoted by  $(\phi_{n})_{n\geq 1}$ such that $\|\phi_{n}\|_{L^{2}(\R^{d})}=1$ and so that for some $M>0$ and all   $n\geq 1$, 
\begin{equation}\label{new} 
\|\phi_{n}\|_{L^{\infty}(\R^{d})}\leq M\lambda_{n}^{-\frac{d}4} (1+\log \lambda_{n})^{1/2}.
\end{equation}
 \end{theo}

 We refer to Theorem~\ref{theoBase} for a more quantitative result, and where we prove that for a natural probability measure, almost all Hermite basis satisfies the property of Theorem~\ref{theoB} (see also Corollary~\ref{coroBase}). For the proof of this result, we need the finest randomisation with $\delta=1$ and $D=2$ in~\eqref{spect1}, so that  ${\bf P}_{\gamma,h}$ is a probability measure on each eigenspace. 

The result of Theorem \ref{theoB} does not hold true in dimension $d=1$. Indeed, in this case one can prove the optimal bound (see \cite{KOTA})
\begin{equation}\label{1D}
\|\phi_{n}\|_{L^{\infty}(\R)}\leq Cn^{-1/12}.
\end{equation}

Let us compare \eqref{dec} with the general known bounds on Hermite functions. We have $H\phi_{n}=\lambda_{n} \phi_{n}$, with $\lambda_{n}\sim c n^{1/d}$, therefore \eqref{dec} can be rewritten
\begin{equation}\label{dec}
\|\phi_{n}\|_{L^{\infty}(\R^{d})}\leq M n^{-1/4}(1+\log n)^{1/2}.
\end{equation}
 For a general basis with $d\geq 2$, Koch and Tataru \cite{KOTA} (see also \cite{KTZ}) prove that 
\begin{equation*} 
\|\phi_{n}\|_{L^{\infty}(\R^{d})}\leq C\lambda_{n}^{{\frac{d}4-\frac1{2}}},
\end{equation*}
which shows that \eqref{new} induces a gain of $d-1$ derivatives compared to the general case.   We stress that we don't now any explicit example of $(\phi_{n})_{n\geq 1}$ which satisfy the conclusion of the Theorem. For instance, the basis $(\phi^{\intercal}_{n})_{n\geq 1}$ obtained by tensorisation of the 1D basis does not realise \eqref{dec} because of~\eqref{1D} which implies the optimal bound
\begin{equation*} 
\|\phi^{\intercal}_{n}\|_{L^{\infty}(\R^{d})}\leq C\lambda_{n}^{-1/12}.
\end{equation*}
Observe also that the basis of radial Hermite functions does not satisfy~\eqref{dec} in dimension $d\geq 2$. As in \cite[Théorème 8]{bule}, it is likely that the   $\log$ term in \eqref{dec} can not be avoided.

 \subsection{Notations and plan of the paper} 
 
 \begin{enonce*}{Notations}
 In this paper $c,C>0$ denote constants the value of which may change
from line to line. These constants will always be universal, or uniformly bounded with respect to the other parameters.\\ 
We denote by $H=-\Delta+|x|^{2}=\sum_{j=1}^{d}(-\partial^{2}_{j}+x^{2}_{j})$ the harmonic oscillator on $\R^{d}$, and   for $s\geq 0$ we define the   Sobolev space $\H^{s}$  by the norm  $\|u\|_{\H^{s}}=\|H^{s/2}u\|_{L^{2}(\R^{d})}\approx \|u\|_{H^{s}(\R^{d})}+\|\<x\>^{s}u\|_{L^{2}(\R^{d})}$.  More generally, we define the spaces $\W^{s,p}$ by the norm $\|u\|_{\W^{s,p}}=\|H^{s/2}u\|_{L^{p}(\R^{d})}$. We write $L^{r,s}(\R^d) = L^r(\R^d, \<x\>^s\text{d}x)$, and its norm $\|u\|_{r,s}$. 
\end{enonce*}

The rest of the paper is organised as follows. In Section \ref{Sect2} we describe  the general probabilistic setting and we prove large deviation estimates on   Hilbert spaces.   In  Section \ref{Sect3} we state crucial estimates on the spectral function of the harmonic oscillator.    Section \ref{Sect4} is devoted to the proof of   weighted Sobolev estimates and of the mains results. In Section \ref{Sect5} we prove Theorem \ref{theoB}.
 
\begin{acknowledgements}
The authors thank Nicolas Burq for discussions on this subject and for his  suggestion to introduce conditions \eqref{condi1}-\eqref{condi0}.
\end{acknowledgements}


\section{A general setting for probabilistic smoothing estimates}\label{Sect2}
Our aim  in this section is to unify several probabilistic approaches to improve smoothing estimates 
established for dispersive equations. This setting is inspired by   papers of Burq-Lebeau \cite{bule},
 Burq-Tzvetkov \cite{BT2,BT3} and their collaborators.

\subsection{The concentration of measure property}

\begin{defi}\ph\label{defiCM}
We say that a family of  Borelian probability measures   $(\nu_{N}, \R^{N})_{N\geq 1}$  satisfies the concentration of measure property if there exist  constants $c,C>0$ independent of $N\in \N$ such that for all Lipschitz and convex function $F : \R^{N}\longrightarrow \R$
\begin{equation}\label{Blip}
\nu_{N}\Big[\,X\in \R^{N}\;:\;\big|F(X)-\E(F(X))\big|\geq r \,\Big]\leq c\,e^{-\frac{Cr^{2}}{ \|F\|^{2}_{Lip}}},\quad\forall r>0, 
\end{equation}
where $\|F\|_{Lip}$ is the best constant so that $\dis |F(X)-F(Y)|\leq \|F\|_{Lip}\|X-Y\|_{\ell^{2}}$.
\end{defi}
For a comprehensive study of these phenomena, we refer to the book of Ledoux \cite{led}. Notice that one of the main features of \eqref{Blip} is that the bound is independent of the dimension of space, which enables to take $N$ large.

Typically, in our applications, $F$ will be a norm in $\R^N$.

\ligne

Let us give some significative examples of such measures.\\[4pt]
\indent $\bullet$ If $(\nu_{N}, \R^{N})_{N\geq 1}$ is   a family of probability measures which satisfies a Log-Sobolev estimate with constant $C^{\star}>0$, then \eqref{Blip} is satisfied for all Lipschitz function $F : \R^{N}\longrightarrow \R$ (see  \cite[Th\'eor\`eme~7.4.1, page 123]{Panorama}).
 Recall that a probability measure $\nu_{N}$ on $\R^{N}$ satisfies a Log-Sobolev estimate if there exists $C>0$  independent of $N\geq 1$ so that for all $f\in \mathcal{C}_{b}(\R^{N})$ 
\begin{equation}\label{logSobo}
\int_{\R^{N}} f^{2}\ln \big( \frac{f^{2}}{\E(f^{2})}\big ) \text{d}\nu_{N}(x)    \leq C\int_{\R^{N}}|\nabla f|^{2}\text{d}\nu_{N}(x), \quad \E(f^{2})=\int_{\R^{N}} f^{2} \text{d}\nu_{N}(x). 
\end{equation}
 Such a property is usually difficult to check. See \cite{Panorama} for more details. Notice that the convexity of $F$ is not needed.\\[4pt]
 \indent$\bullet$  A probability measure of the form $\text{d}\nu_{N}(x)=c_{\a,N}\exp\big({-\sum_{j=1}^{N}|x_j|^{\a}}\big)\text{d}x$, $x \in\R^N$, satisfies \eqref{Blip} if and only if $\a\geq 2$ (see \cite[page 109]{Panorama}).
\\[4pt]
 \indent$\bullet$ Assume that $\nu$ is a measure on $\R$ with bounded support, then $\nu_{N}=\nu^{\otimes N}$ satisfies the concentration of measure property. This is the Talagrand theorem \cite{Tala} (see also \cite{Tao} for an introduction to the topic).

\begin{assumption}\ph\label{Assum1}
Consider a   probability space
$(\Omega, {\cal F}, {\P})$  and  let $\{X_n,\;n\geq1\}$ be a sequence of independent, identically distributed, real or complex random values. In the sequel we can assume that they are real with the identification $\C\approx \R^{2} $.  Moreover, we assume that for all $n\geq 1$, 
\begin{enumerate}[(i)]
\item  Denote by $\nu$  law  of the  $X_{n}$. We assume that the family $(\nu^{\otimes N}, \R^{N})_{N\geq 1}$ satisfies the concentration of measure property in the sense of Definition~\ref{defiCM}.
\item The r.v. $X_{n}$ is centred: $\E(X_{n})=0$. 
\item The r.v. $X_{n}$ is normalized: $\E( X^{2}_{n})=1$.
\end{enumerate}
\end{assumption}

Under Assumption \ref{Assum1}, for all $n\geq 1$, and $\eps>0$ small enough
\begin{equation}\label{exp2}
\E( \e^{\eps X^{2}_{n}}  )<+\infty.
\end{equation}
Indeed, by Definition \ref{defiCM} with $F(X)=X_{n}$
\begin{equation*}
\E( \e^{\eps X^{2}_{n}}  )=\int_{0}^{+\infty}\nu(\,\e^{C
\eps X_{n}^{2}}>\lambda\, )\text{d}\lambda=1+\int_{1}^{+\infty}\nu \big(\,|X_{n}|>\sqrt{\frac{\ln \lambda}{\eps}}\, \big)\text{d}\lambda\leq1+ 2\int_{1}^{+\infty}\lambda^{-\frac1{\eps C}}\text{d}\lambda<+\infty.
\end{equation*}
Next, with the inequality $s|x|\leq \eps x^{2}/2+s^{2}/(2\eps)$, we obtain that for all $s\in \R$, $\E( \e^{s X_{n}}  )\leq C\e^{Cs^{2}}$ which in turn implies (see \cite[Proposition 46]{poiret1}) that there exists $C>0$ so that for all $s\in \R$
\begin{equation}\label{gcond}
\E( \e^{s X_{n}}  )\leq    \e^{Cs^{2}}.
\end{equation}

\begin{rema}
Condition \eqref{gcond} is weaker that \eqref{logSobo}: a family of independent centred r.v. ${\{X_n,\;n\geq1\}}$  which satisfies \eqref{gcond} does not necessarily satisfy \eqref{Blip} for all Lipschitz function $F$.  Indeed, using Kolmogorov estimate, one can prove (see \cite{led}) that condition \eqref{Blip} is equivalent to  
\begin{equation}\label{ExpLip}
\int_{\R^d} {\rm e}^{sF}d\nu \leq {\rm e}^{C s^2\|F\|^{2}_{Lip}}, \quad \forall\, s\in \R,
\end{equation}
for all Lipschitz function $F$ with $\nu$-mean 0. 

\end{rema}

We conclude with  the elementary property
\begin{lemm}\ph\label{tech1}
Assume that  $\{X_n\}$ satisfies (\ref{gcond}) and that $\{\alpha_j,\; 1\leq j\leq N\}$ are real numbers such that 
$\di{\sum_{1\leq j\leq N}\alpha_j^2 \leq 1}$. Then $\di{X:=\sum_{1\leq j\leq N}\a_{j} X_j}$ satisfies (\ref{gcond})
with the same constant $C$.
\end{lemm}
\begin{proof}
It is a direct application of  \eqref{ExpLip} with $\dis F(X)=\sum_{j=1}^{N}\a_{j}X_{j}$. 
\end{proof}
\subsection{Probabilities on Hilbert spaces}
In this sub-section ${\cal K}$ is a separable complex Hilbert space and $K$ is a self-adjoint, positive operator on~${\cal K}$
with a compact resolvent. We denote by ${\{\varphi_j,\;j\geq 1\}}$ an orthonormal basis of eigenvectors  of~$K$,
$K\phi_j =\lambda_j\varphi_j$, and $\{\lambda_j,\;j\geq 1\}$ is the non decreasing  sequence of eigenvalues
of $K$ (each is repeated according to its multiplicity). Then we get  a natural scale of Sobolev
spaces associated with $K$  defined for $s\geq 0$ by ${\cal K}^s = {\rm Dom}(K^{s/2})$.

Now we want to introduce probability measures  on these spaces and on some  finite dimensional
spaces of ${\cal K}$.

Let us describe in our setting the randomization technique  deeply used   by  Burq-Tzvetkov  in~\cite{BT2}.
Let $\gamma = \{\gamma_j\}_{j\geq 1}$ a sequence of complex numbers such
that $\di{\sum_{j\geq 1}\lambda_j^{s}\vert\gamma_j\vert^2 < +\infty}$. 

Consider a   probability space
$(\Omega, {\cal F}, {\P})$  and  let $\{X_n,\;n\geq1\}$ be  independent, identically distributed random variables  which satisfy  Assumption \ref{Assum1}.

We denote by 
$\di{v_\gamma^0 =  \sum_{j\geq 1}\gamma_j \varphi_j} \in{\cal K}^s$, and we define the random vector
$\di{v_\gamma(\omega) = \sum_{j\geq 1}\gamma_jX_j(\omega)\varphi_j}$.
We have $\E(\Vert v_\gamma\Vert^{2}_{\cal K}) < +\infty$, therefore  $v_\gamma\in{\cal K}^s$, a.s.
We define the measure 
$\mu_\gamma$   on ${\cal K}^s$ as the law of  the random vector $v_\gamma$.
\subsubsection{\bf The Kakutani theorem} The following proposition
gives some properties of the measures~$\mu_{\gamma}$ (see \cite{BTproba} for more details).
\begin{prop}\ph\label{kakut}
Assume that all  random variables $X_j$ have  the same law $\nu$.
\begin{enumerate}[(i)]
\item  If the support of $\nu$ is $\R$ and if $\gamma_j\neq 0$ for all $j\geq 1$ then
the support of $\mu_\gamma$ is ${\cal K}^s$.
\item If for some $\epsilon >0$ we have  $v_\gamma^0\notin {\cal K}^{s+\epsilon}$ then $\mu_\gamma( {\cal K}^{s+\epsilon})=0$.
\item Assume that we are in the particular case where $\text{d}\nu(x)=c_{\alpha}\e^{-|x|^{\alpha}}\text{d}x$ with $\a\geq 2$. Let $\gamma=\{\gamma_j\}$ and $\beta=\{\beta_j\}$ be  two complex sequences and assume
that 
\beq\label{critka}
\sum_{j\geq 1}\left(\left\vert\frac{\gamma_j}{\beta_j}\right\vert^{\a/2} - 1\right)^2 = +\infty.
\eeq
\end{enumerate}
Then the measures $\mu_\gamma$ and $\mu_\beta$ are mutually singular,  i.e there exists a measurable set 
$A\subset {\cal H}^s$ such that $\mu_\gamma(A) = 1$ and $\mu_\beta(A) = 0$.
\end{prop}
We give the proof of $\it (iii)$ in Appendix \ref{AppA}.\ligne

We shall see now that condition (\ref{condi1}) (resp. (\ref{condi0})) can be perturbed so that  Proposition \ref{kakut} gives us an infinite number of mutually singular measures on ${\cal K}^s$.

\begin{lemm}
Let  $\gamma$ satisfying (\ref{condi1}) (resp. (\ref{condi0}))  and $\delta = \{\delta_n\}_{n\geq 1}$
such that $\vert\delta_n\vert \leq \varepsilon\vert\gamma_n\vert$ for every $n\geq n_0$. 
Then for every $\varepsilon\in[0, \sqrt 2-1[$, the sequence $\gamma +\delta$ satisfies 
(\ref{condi1}) (resp. (\ref{condi0}))  (with new constants).
\end{lemm}
We do not give the details of the proof. From this Lemma and Proposition \ref{kakut}  we get an infinite number of  measures
$\mu_\gamma$ with $\gamma$ satisfying~(\ref{condi1}) (resp. (\ref{condi0})).
Let  $\epsilon_j$ be any sequence  such that $\di{\sum_{j\geq 1} \epsilon_j^2 =+\infty}$  and 
$\limsup\epsilon_j < \sqrt 2-1$  and denote by $\epsilon\otimes\gamma$
the sequence $\epsilon_j\gamma_j$.  Then $\mu_\gamma$ and $\mu_{\gamma+\epsilon\cdot\gamma}$ are mutually singular.\ligne

\subsubsection{\bf Measures on the sphere ${\bf S}_h$}
Now we consider finite dimensional subspaces ${\cal E}_h$  of ${\cal K}$  defined by spectral localizations depending on a small parameter $0<h\leq 1$ ($h^{-1}$ is a measure of   energy for the quantum Hamiltonian $K$). In the sequel, we use the notations  $I_h = [\frac{a_h}{h}, \frac{b_h}{h}[$, $N_h$, $\Lambda_{h}$ and ${\cal E}_h$ introduced in Section \ref{intro}, and we  assume that \eqref{spect1} is satisfied. Observe that ${\cal  E}_h$ is the spectral subspace of  $K$ in the interval $I_h$: ${\cal  E}_h =\Pi_h{\cal K}$ where $\Pi_h$ is the orthogonal projection on  ${\cal K}$.  For simplicity,  we  sometimes denote  by $N=N_h$, $\Lambda=\Lambda_h$, \dots, with implicit dependence in $h$. Our goal is to find uniform estimates in $h\in]0,h_0[$
  for a small constant $h_0>0$.\ligne

Let us consider the random vector  in ${\cal E}_h$
\beq\label{Evec}
v_{\gamma}(\omega) : = v_{\gamma,h} (\omega)= \sum_{j\in\Lambda}\gamma_jX_j(\omega)\varphi_j,
\eeq
and assume that  \eqref{condi0} is satisfied. In the sequel we denote by  $\vert\gamma\vert_\Lambda^2= \di{\sum_{n\in\Lambda}\gamma_n^2}$.

Now we consider probabilities on the unit sphere 
${\bf S}_h$ of the subspaces
 ${\cal E}_h$.  The random vector $v_{\gamma}$ in~\eqref{Evec} defines a probability measure $\nu_{\gamma,h}$ on  ${\cal E}_h$. Then we can define a probability measure ${\bf P}_{\gamma,h}$    on~${\bf S}_h$ as the image of 
 by $v\mapsto \frac{v}{\Vert v\Vert}$. Namely,  we have for every Borel and bounded function $f$ on~${\bf S}_h$, 
 \beq\label{immes}
 \int_{{\bf S}_h}f(w){\bf P}_{\gamma,h}(\text{d}w) = \int_{{\cal E}_h}f\left(\frac{v}{\Vert v\Vert_{\cal K}}\right)\nu_{\gamma,h}(\text{d}v) = \int_\Omega f\left(\frac{v_\gamma(\omega)}{\Vert v_\gamma(\omega)\Vert_{\cal K}}\right)\P(\text{d}\omega).
 \eeq
 Remark that we have
 $$
 \Vert v_\gamma(\omega)\Vert_{\cal K}^2 =
 \sum_{j\in\Lambda}\vert\gamma_j\vert^2\vert X_j(\omega)\vert^2
 $$
 and 
 $$
 \E(\Vert v_\gamma\Vert_{\cal K}^2) = \sum_{j\in\Lambda}\vert\gamma_j\vert^2=|\gamma|^{2}_{\Lambda}.
 $$
Let us detail  two particular cases of interest:\ligne

$\bullet$ If    $\vert\gamma_n\vert =\frac{1}{\sqrt N}$ for  all $j\in\Lambda$
 and if $X_n$ follows the complex normal law ${\cal N_{\C}}(0,1)$ then ${\bf P}_{\gamma,h}$ is the uniform
 probability on ${\bf S}_h$ considered in \cite{bule}. This follows from (\ref{immes}) and property of Gaussian laws.\\[2pt]

$\bullet$ Assume that for all $n\in \N$, $\P(X_n=1)=\P(X_n=-1)=1/2$, then ${\bf P}_{\gamma,h}$ is a convex sum of~$2^N$ Dirac measures. Indeed we have $  \Vert v_\gamma(\omega)\Vert_{\cal K}^2 =\sum_{j\in\Lambda}\vert\gamma_j\vert^2=|\gamma|^{2}_{\Lambda}$. Denote by $(\eps^{(k)})_{1\leq k\leq 2^N}$ all the sequences so that $\eps_j^{(k)}=\pm 1$ for all $1\leq j\leq N$, and set 

\begin{equation*}
\Phi_k= \frac1{\vert \gamma\vert}\sum_{j\in\Lambda}\gamma_j \eps^{(k)}_j\varphi_j,\quad 1\leq k\leq 2^N.
\end{equation*}
Then 
\begin{equation*}
{\bf P}_{\gamma,h}=\frac1{2^N}\sum_{k=1}^{2^N}\delta_{\Phi_k}.
\end{equation*}

\ligne
To get an optimal lower bound for $L^\infty$  estimates  we shall need a stronger normal concentration estimate than estimate given  in (\ref{Blip}).  Hence we make the following assumptions:
\begin{assumption}\ph\label{Assum2} We assume that 
\begin{enumerate}[(i)]
\item The random variables $X_{j}$ are standard independent Gaussians ${\cal N}_{\C}(0,1)$.
\item The sequence $\gamma$ satisfies (\ref{condi0}).
\end{enumerate}
\end{assumption}

 Let  $L$ be  a linear  form on ${\cal E}_h$, and denote by $e_L = \di{\sum_{j\in\Lambda_{h}}\vert L(\varphi_j)\vert^2}$. The main result of this section is the following
 
 \begin{theo}\ph\label{thm28}
 Let  $L$ be  a linear  form on ${\cal E}_h$.  Suppose that (\ref{condi1}) holds and that   Assumption \ref{Assum1} is satisfied. Then  there exist   $C_{2},c_{2}>0$   so that 
 \beq\label{psph}
 {\bf P}_{\gamma,h}\Big[u\in{\bf S}_h:\; |L(u)|\geq t\Big] \leq C_2{\rm e}^{-c_2\frac{N}{e_L  }t^2},\quad \; \forall  \, t\geq 0,\;\;  \forall  \,h\in]0,  h_0],
  \eeq
Moreover, if  (\ref{condi0}) holds, there exist $C_{1},c_{1}>0$ and $\eps_{0},h_0>0$ so that   
  \begin{equation}\label{minor*}
C_1{\rm e}^{-c_1\frac{N}{e_L}t^2} \leq  {\bf P}_{\gamma,h}\Big[u\in{\bf S}_h:\; |L(u)|\geq t\Big]  ,\quad \forall t\in \big[0, \eps_{0}\frac{\sqrt{e_{L}}}{\sqrt{N}}\big],\quad \forall h\in]0, h_0].
\end{equation}

Furthermore, if Assumption \ref{Assum2} is satisfied, there exist $C_{1},C_{2},c_{1},c_{2},\eps_{0}, h_0>0$ so that 
\begin{equation}\label{Binf}
 C_{1}\,{\rm e}^{-c_{1}\frac{N}{e_L }t^2}\leq  {\bf P}_{\gamma,h}\Big[u\in{\bf S}_h:\; |L(u)|\geq t\Big] \leq C_{2}\,{\rm e}^{-c_{2}\frac{N}{e_L }t^2},\quad \; \forall \, t\in [0, \eps_{0}\sqrt{e_{L}}\,],\;\;  \forall  \,h\in]0,  h_0].
\end{equation}
 \end{theo}
 Since $ {\bf P}_{\gamma,h}$ is  supported by  ${\bf S}_h$, the bounds in the previous result don't depend on $|\gamma|_{\Lambda}$. The restriction on $t\geq 0$ in \eqref{Binf} is natural,  because by  the Cauchy-Schwarz inequality we have
  \begin{equation*}
  \vert L(u)  \vert\leq \sqrt{e_L},\quad  \forall u\in{\bf S}_h.
 \end{equation*}

In the applications we give, there is some embedding   ${\cal K}^s\rightarrow C(M)$, for $s>0$ large enough, where
$M$ is a metric space. We have ${\cal E} \subseteq \bigcap_{s\in\R} {\cal K}^s$, thus  we can consider the Dirac evaluation linear form $\delta_x(v) = v(x)$. In this case
we have $\di{e_L = \sum_{j\in \Lambda}\vert\varphi_j(x)\vert^2 =e_x}$, which is usually called the spectral function of $K$ in the interval $I$.  

For example, one can consider the Laplace-Beltrami operator  on compact Riemannian manifolds, namely $K=-\triangle$ and ${\cal K}^s={\cal H}^s(M)$ are the usual Sobolev spaces: this is the framework of \cite{bule}. In Section \ref{Sect3} we will apply the result of Theorem \ref{thm28} to the Harmonic oscillator $K=-\triangle +\vert x\vert^2$ on $\R^d$. In this latter case~${\cal K}^s$
  is the weighted Sobolev space 
  $$
  {\cal K}^s =\big\{\,u\in {\cal H}^s(\R^d),\; \vert x\vert^su\in L^2(\R^d)\,\big\},\; s\geq 0.
  $$

 \begin{rema}
In the particular case where  ${\bf P}_{\gamma,h}$ is the uniform probability  on  ${\bf S}_h$, we have the explicit   computation
$$
 {\bf P}_{\gamma,h}\Big[u\in{\bf S}_h:\; |L(u)|\geq t\Big]  = \Phi\left(\frac{t}{ \sqrt{e_{L}}}\right),
$$
 where 
  \begin{equation}\label{phi}
  \Phi(t) = \1_{[0,1[}(t)(1-t^2)^{N-1},
\end{equation}
and \eqref{Binf} follows directly. For a proof of \eqref{phi}, see   \cite{bule}   or in Appendix \ref{AppC} of this paper for an alternative argument. 
\end{rema}

For the proof of Theorem \ref{thm28} we will need the following result. 
\begin{prop}\ph\label{L}
Assume that $\gamma$ satisfies \eqref{condi1}. Let  $L$ be  a linear  form on ${\cal E}_h$. Then we have the large deviation estimate
\begin{equation*}
{\P}\Big[\,\om\in \Omega:\vert L(v_\gamma)\vert \geq t\,\Big] \leq 4{\rm e}^{-\kappa_1\frac{N}{e_L\vert\gamma\vert_\Lambda^2}t^2},
\end{equation*}
where $\kappa_1 =\frac{\kappa_0}{4K_1}$. As a consequence, if $\nu_{\gamma, h}$ denotes the probability law of $v_{\gamma}$, then 
\begin{equation*}
\nu_{\gamma,h}\Big[\,w \in \mathcal{E}_h:\;\vert L(w)\vert \geq t\,\Big] \leq 4{\rm e}^{-\kappa_1\frac{N}{e_L\vert\gamma\vert_\Lambda^2}t^2}.
\end{equation*}
\end{prop}
\begin{proof}
We have
$$
L(v_\gamma) = \sum_{j\in\Lambda}\gamma_nX_n(\om)L(\varphi_n).
$$
It is enough to assume that $L(v_\gamma)$ is real  and to estimate
${\P}\big[\,\om\in \Omega:L(v_\gamma)\geq t\,\big]$. Using the Markov inequality, we have for all $s>0$
$$
{\P}\big[L(v_\gamma)\geq t\big] \leq {\rm e}^{-st}{\E}({\rm e}^{sL(v_\gamma)}),
$$
and thanks to \eqref{condi1} we have
 $$
 \sum_{j\in\Lambda}\vert\gamma_j L(\varphi_j) \vert^2 \leq K_1\frac{\vert\gamma\vert_\Lambda^2}{N}\sum_{j\in\Lambda}\vert L(\varphi_j)\vert^2.
 $$
  Using Lemma \ref{tech1} we get
 \begin{equation*}
  {\P}\big[L(v_\gamma)\geq t\big] \leq {\rm e}^{-st}{\rm e}^{\kappa_0K_1\frac{e_L}{N}\vert\gamma\vert_\Lambda^2s^2},
 \end{equation*}
and with the choice $s = \frac{\kappa_0}{2}\frac{tN}{K_1e_L\vert\gamma\vert_\Lambda^2}$  we obtain $\dis   {\P}\big[L(v_\gamma)\geq t\big] \leq {\rm e}^{-\kappa_1\frac{N}{\vert\gamma\vert_\Lambda^2e_L}t^{2}}.$
\end{proof}

 It will be useful to show that  $\Vert v_\gamma(\omega)\Vert_{\cal K}^2 $  is close to its expectation    for large $N$.
 \begin{lemm}\ph\label{LD1}
Let $\gamma$  satisfying the squeezing condition  (\ref{condi0}).  Then then exists $c_{0}>0$ (depending only on $K_{0}$ and $K_{1}$) such that for every $\epsilon >0$ 
\begin{equation*} 
 \P\Big[\,\om \in \Omega:\,\big\vert\Vert v_\gamma(\omega)\Vert_{\cal K}^2 -\vert\gamma\vert_{\Lambda}^2\big\vert >\epsilon\,\Big]\leq 2 {\rm e}^{-\frac{\eps c_0  N}{\vert\gamma\vert_{\Lambda}^2}}.
\end{equation*}
\end{lemm}
 \begin{proof}
 It is enough to consider the  real case, so we assume that $\gamma_n$ and $X_n$ are
 real  and $\{X_n, n\geq 1\}$ have a common law  $\nu$. We also assume that   $\vert\gamma\vert_{\Lambda}^2=1$.
 
We have
 $$
 \Vert v_\gamma(\omega)\Vert_{\cal K}^2 = \sum_{j\in\Lambda}\vert\gamma_j\vert^2X_j^2(\omega): =M_N(\omega).
 $$
  From large number law, $\Vert v_\gamma(\omega)\Vert_{\cal K}^2$  converges to 1 a.s.   To  estimate the tail  we use the Cramer-Chernoff large deviation principle (see e.g. \cite[ $\S$ 5, Chapter IV]{Shir}).  This applies because  from (\ref{exp2}) we know that
  $f(s) :=\E({\rm e}^{sX_1^2})$  is $\mathcal{C}^2$  in $]-\infty, s_0[$ for some $s_0>0$.  
  
  We reproduce here a well known  computation in large deviation theory. Define the  cumulant function $g(s) =\log(f(s))$ which  is well defined for $s<s_0$.   Now, since the $X_{j}$ are i.i.d.,  for $t,s\geq 0$   we have
  \begin{eqnarray*}
    \P\big[M_N >t\big] = \P\big[{\rm e}^{sNM_N}>{\rm e}^{sNt}\big]& \leq&\E({\rm e}^{sNM_N}){\rm e}^{-sNt}\\
    &=&\prod_{j\in \Lambda} \e^{-(Ns|\gamma_{j}|^{2}t-g(Ns|\gamma_{j}|^{2}))}.
  \end{eqnarray*}
  Next, apply the Taylor formula to $g$ at 0: $g(0)=0$, $g'(0) = 1$ so
   $t\tau-g(\tau) = (t-1)\tau + {\cal O}(\tau^2)$, hence there exists $s_1>0$ such that for $0\leq \tau\leq s_{1}$,
   $t\tau -g(\tau) \geq (t-1)\frac{\tau}{2}$.  Then, with $t=1+\eps$, and since $N|\gamma_{j}|^{2} \leq K_{0}$ we get
  \begin{equation*}
      \P\big[M_N > 1+\epsilon\big]   \leq \prod_{j\in \Lambda}\e^{-\eps Ns|\gamma_{j}|^{2}/2}=\e^{-\eps s N/2},
  \end{equation*}
provided $s>0$ is small enough, but independent of $\eps>0$ and $N\geq 1$.   The same computation applied to $-M_N$ gives  as well
   $
   \dis   \P\big[M_N <1-\epsilon\big]   \leq {\rm e}^{-\eps c_0 N}.
$
  \end{proof}


 \begin{proof}[Proof of \eqref{psph}]
By homogeneity, we can assume that $|\gamma|_{\Lambda}=1$. Denote by 
\begin{equation}\label{defA}
A=\big\{ \,\om \in \Omega:\,\big\vert\Vert v_\gamma(\omega)\Vert_{\cal K}^2 -1\big\vert \leq 1/2\,\big\}.
\end{equation}
By the Cauchy-Schwarz inequality, for all $u\in {\bf S}_{h}$, we obtain $|L(u)|\leq e^{1/2}_{L}$. Thus in the sequel we can assume that $t\leq e^{1/2}_{L}$. Then, from Proposition \ref{L} and Lemma \ref{LD1} we have
\begin{multline}\label{decomp}
{\bf P}_{\gamma,h}\Big[u\in{\bf S}_h:\; |L(u)|\geq t\Big]=  \P\big[\,\om \in \Omega:\,|L(v(\om))|\geq t \|v(\om)\|_{L^{2}}\big]  \\
= \P\big[\, (|L(v(\om))|\geq t \|v(\om)\|_{L^{2}})\cap A\big]+\P\big[\, (|L(v(\om))|\geq t \|v(\om)\|_{L^{2}})\cap A^{c}\big].
  \end{multline}
 Therefore 
 \begin{eqnarray*}
{\bf P}_{\gamma,h}\Big[u\in{\bf S}_h:\; |L(u)|\geq t\Big] &\leq& \P\big[\, |L(v(\om))|\geq t/2 \big]+\P(A^{c})\\
&\leq &  C_{1}{\rm e}^{-c_{1}\frac{N}{e_L}t^2}+2\e^{-c_2 N}\leq  C{\rm e}^{-c\frac{N}{e_L}t^2},
 \end{eqnarray*}
which implies \eqref{psph}.
\end{proof}

We now turn to the proof of  \eqref{minor*}.  We will need the following result

\begin{lemm}\ph\label{minor} 
We suppose   that $\gamma$ satisfies \eqref{condi0} and that Assumption \ref{Assum1} is satisfied. Then there exist $C_1>0,  c_1>0$, $h_0>0$, $\epsilon_0>0$   such that
\begin{equation*}
\P\Big[\,\om\in\Omega:\vert L(v_\gamma(\omega))\vert \geq t\,\Big] \geq C_1{\rm e}^{-c_1\frac{N}{e_L \vert\gamma\vert_{\Lambda}^2}t^2},\quad \forall t\in \big[0, \eps_{0} \frac{\sqrt{e_{L}}|\gamma|_{\Lambda}}{\sqrt{N}}\big],\quad \forall h\in]0, h_0].
\end{equation*}
\end{lemm}

\begin{proof}
 Let us first recall the  Paley-Zygmund inequality\footnote{We thank Philippe Sosoe for this suggestion.}: Let $Z\in L^2(\Omega)$ be a r.v such that $Z\geq 0$, then for all $0<\lambda<1$,
 \begin{equation}\label{PZ}
   \P\big( Z>\lambda \Vert Z\Vert_1\big)\geq\Big( (1-\lambda)\frac{ \dis  \Vert Z\Vert_1}{ \Vert Z\Vert_2} \Big)^2.
 \end{equation}
 We apply   \eqref{PZ} to the random variable $Z=\vert Y_N\vert^2 $, with 
 \begin{equation*}
 Y_{N}=\frac{\sqrt{N}}{\sqrt{e_{L}}|\gamma|_{\Lambda}}L(v_{\gamma})= \frac{\sqrt{N}}{\sqrt{e_{L}}|\gamma|_{\Lambda}}\sum_{j\in \Lambda} \gamma_{j}X_{j}L(\phi_{j}),
 \end{equation*}
and $\lambda=1/2$. By \eqref{condi0}, we have   $c_{0}\leq \|Y_{N}\|_{2}\leq C_{0}$ uniformly in $N\geq1$. Next, recall the  Khinchin inequality (see {\it e.g.} \cite[Lemma 4.2]{BT2} for a proof) : there exists $C>0$ such that for all real $k\geq 2$ and $(a_{n})\in \ell^{2}(\N)$
\begin{equation*} 
\|\sum_{n\in \Lambda}X_{n}(\om)\,a_{n}\|_{L_{\P}^{k}}\leq C\sqrt{k}\Big(\sum _{n\in \Lambda}|a_{n}|^{2}\Big)^{\frac12}.
\end{equation*}
Therefore, there exists $C_1>0$ such that  $ \|Y_{N}\|_{4}\leq C_1$. As a result, there exist  $\eta>0$ and $\eps>0$ so that    for all $N\geq 1$, $\P(\vert Y_N\vert >\eta)>\eps$, which implies the result.
\end{proof}

\begin{proof}[Proof of \eqref{minor*}]
We assume that $|\gamma|_{\Lambda}=1$, and consider the set $A$ defined in \eqref{defA}. Then by \eqref{decomp} and the inequality $\P(B\cap A)\geq \P(B)-\P(A^{c})$ we get 
 \begin{eqnarray*}
{\bf P}_{\gamma,h}\Big[u\in{\bf S}_h:\; |L(u)|\geq t\Big] &\geq& \P\big[\, (|L(v(\om))|\geq t \|v(\om)\|_{L^{2}})\cap A\Big]\\
&\geq &  \P\big[\, |L(v(\om))|\geq 3t/2  \big]-\P(A^{c})\\
&\geq & C_{1}{\rm e}^{-c_{1}\frac{N}{e_L}t^2}-2\e^{-c_2 N},
 \end{eqnarray*}
 where in the last line we used Lemma \ref{minor} and Lemma \ref{LD1}. This yields the result if $t\leq \eps_{0}\frac{\sqrt{e_{L}}}{\sqrt{N}}$ with $\eps_{0}>0$ small enough.
 \end{proof}

We now prove   \eqref{Binf}. To begin with, we can state
\begin{lemm}\ph\label{lemM} 
We suppose that Assumption \ref{Assum2} is satisfied. Then there exist $C_1>0,  c_1>0$, $h_0>0$, $\epsilon_0>0$   such that
\begin{equation*}
\P\Big[\,\om\in\Omega:\vert L(v_\gamma(\omega))\vert \geq t\,\Big] \geq C_1{\rm e}^{-c_1\frac{N}{e_L \vert\gamma\vert_{\Lambda}^2}t^2},\quad \forall t\geq 0,\; \forall h\in]0, h_0].
\end{equation*}
\end{lemm}
\begin{proof}
 Denote by $\gamma\otimes L(\phi)$ the vector $(\gamma\otimes L(\phi))_{j}=\gamma_{j}L(\phi_{j})$. Observe that, thanks to \eqref{condi0},
\begin{equation*}
K_{1}\frac{\vert\gamma\vert_{\Lambda}^2 e_{L}}{N}\leq |\gamma\otimes L(\phi)|^{2}=\sum_{j\in \Lambda_{h}}\gamma^{2}_{j}|L(\phi_{j})|^{2}\leq K_{0}\frac{\vert\gamma\vert_{\Lambda}^2 e_{L}}{N}.
\end{equation*}
Then, using the rotation invariance of the  Gaussian law and the previous line, we get 
\begin{eqnarray*}
\P\Big[\om\in\Omega:\,| L(v_\gamma(\omega))\vert \geq t\,\Big] &=& \P\Big[\,\Big| \< \frac{\gamma\otimes L(\phi)}{|\gamma\otimes L(\phi)|},X\>\Big|\geq \frac{t}{|\gamma\otimes L(\phi)|} \,  \Big] \\
&=& \frac{1}{\sqrt{2\pi}}
\int_{ |s|\geq \frac{t}{     |\gamma\otimes L(\phi)|}     }{\rm e}^{-s^2/2}\text{d}s \\
&\geq & C{\rm e}^{-\frac{cN}{e_L\vert\gamma\vert_{\Lambda}^2}t^2}.
\end{eqnarray*}
\end{proof}

The estimate \eqref{Binf} then follows from Lemma \ref{lemM} and with the same argument as for Lemma~\ref{minor}.

\subsubsection{\bf Concentration phenomenon} We now state a concentration property for~${\bf P}_{\gamma,h}$, inherited from Assumption \ref{Assum1} and condition \eqref{condi0}. See \cite{led} for more details on this topic.
\begin{prop}\ph\label{SC}
Suppose that the i.i.d. random variables $X_{j}$ satisfy Assumption \ref{Assum1} and suppose that condition \eqref{condi0} is satisfied. Then there exist constants $K>0$, $\kappa >0$ (depending only on~$C^{\star}$) such that for every Lipschitz function $F: {\bf S}_h \longrightarrow \R$   satisfying
$$
\vert F(u) -F(v)\vert \leq \|F\|_{Lip} \|u-v\|_{L^{2}(\R^{d})},\quad  \forall u, v\in{\bf S}_h,
$$  
 we have
\beq\label{supconcent}
{\bf P}_{\gamma,h}\Big[ \,u\in {\bf S}_h : \vert F-{\cal M}_F\vert > r\,\Big] \leq K{\rm e}^{-\frac{\kappa Nr^2}{\|F\|^{2}_{Lip}}},\quad  \forall r>0,\; h\in]0, 1],
\eeq
where ${\cal M}_F$ is a median for $F$.
\end{prop}

Recall that a median ${\cal M}_F$ for $F$ is defined by 
\begin{equation*}
{\bf P}_{\gamma,h}\big[ \,u\in {\bf S}_h : F \geq {\cal M}_F  \,\big]\geq \frac12,\quad {\bf P}_{\gamma,h}\big[ \,u\in {\bf S}_h : F \leq {\cal M}_F  \,\big]\geq \frac12.
\end{equation*}

In Proposition \ref{SC},  the distance in $L^{2}$ can be replaced with the geodesic distance   $d_S$ on ${\bf S}_h$, since we can check that 
$$\|u-v\|_{L^{2}(\R^{d})}\leq d_{S}(u,v)=2\arcsin\big(\frac{\|u-v\|_{L^{2}(\R^{d})} }2\big)\leq \frac{\pi}2\|u-v\|_{L^{2}(\R^{d})}.$$

When  ${\bf P}_{\gamma,h}$  is the uniform probability on ${\bf S}_h$, Proposition \ref{SC}  is proved in \cite[Proposition 2.10]{led}, and the proof can be adapted in the general case (see Appendix \ref{AppD}). The factor $N$ in the  exponential  of r.h.s of (\ref{supconcent}) will be crucial in our application.


\section{Some spectral estimates for the harmonic oscillator}\label{Sect3}
Our goal  here is to  apply the  general setting  of Section \ref{Sect2} to  the harmonic oscillator in $\R^d$. This way
 we shall get   probabilistic estimates  analogous to results proved in~\cite{bule}  for the
 Laplace operator in a compact Riemannian manifold. 
 
 In the following, we consider the Hamiltonian $ H = -\triangle + V(x)$
  with $V(x) = \vert x\vert^2$, $x\in\R^d$  for  $d\geq 2$.  For this model,  all  the  necessary spectral  estimates are already known. More general confining potentials $V$ shall be considered in the forthcoming paper \cite{PRT3}.\ligne

  A first and basic  ingredient in  probabilistic approaches of weighted Sobolev spaces is a   good knowledge 
concerning the asymptotic behavior  of eigenvalues and eigenfunctions of  $H$. The eigenvalues of this operator are the $\big\{2(j_{1}+\dots+j_{d})+d,\; j\in \N^{d}\big\}$, and we can order them in a  non decreasing sequence $\{\lambda_j, \;j\in\N\}$, repeated according to their multiplicities. We denote by  $\{\phi_{j}, j\in\N\}$ an orthonormal basis in $L^2(\R^d)$ of eigenfunctions (the Hermite functions), so that $H \phi_{j} =\lambda_j \phi_{j}$.
The spectral function is then defined as 
$\di{\pi_{H}(\lambda;x,y) = \sum_{\lambda_j\leq \lambda}\phi_{j}(x)\overline{\phi_{j}(y)}}$ (recall that this definition does not depend on the choice of $\{\phi_{j}, j\in\N\}$). When the energy $\lambda$ is localized in $I\subseteq \R^{+}$ we denote
 by  $\Pi_H(I)$  the spectral projector of  $ H$ on  $I$.
 The range  ${\cal E}_H(I)$ of  $\Pi_H(I)$ is spanned  by  $\{\phi_{j}; \lambda_j\in I\}$ and $\Pi_H(I)$ 
has an  integral kernel given by 
 $$
 \pi_H(I;x,y) = \sum_{[j\,:\,\lambda_j\in I]}\phi_{j}(x)\overline{\phi_{j}(y)}.
 $$ 
We will also use the notation ${\cal E}_H(\lambda)= {\cal E}_H([0, \lambda])$, 
$N_H(\lambda)={\rm dim}[{\cal E}_H(\lambda)]$.

\subsection{Interpolation inequalities} We begin with some general interpolation results which will be needed in the sequel.  In $\R^d$, the spectral function   $\pi_{H}(\lambda; x,x)$ is fast decreasing for  $\vert x\vert\rightarrow +\infty$  so
  it is natural to work with weighted $L^{p}$ norms.   We denote by   $\<x\>^s=(1+\vert x\vert ^2)^{s/2}$ and  introduce the  following Lebesgue space with weight  
 $$
 L^{p,s}(\R^d) = \left\{u,\;{\rm Lebesgue \;\; measurable}\, : \int \vert u(x)\vert^p\<x\>^s\text{d}x < +\infty\right\} = L^p(\R^d, \<x\>^s\text{d}x),
 $$
 endowed with its natural norm, which we denote by $\|u\|_{p,s}$. For $p=\infty$, we set $\di{\Vert u\Vert_{\infty, s} = \sup_{x\in\R^d}\<x\>^s\vert u(x)\vert}$.

The following interpolation inequalities hold true.  Let $1\leq p_1\leq p \leq p_{0}\leq +\infty$ and $\kappa\in ]0, 1[$ such that
 $\frac{1}{p} =\frac{\kappa}{p_1} + \frac{1-\kappa}{p_0} $. Then for $p_{0}<+\infty$ we have
\begin{equation}\label{interpol1}
 \Vert u\Vert_{L^{p,s}(\R^d)} \leq  (\Vert u\Vert_{L^{p_0,s_0}(\R^d)})^{1-\kappa}(\Vert u\Vert_{L^{p_1,s_1}(\R^d)})^{\kappa}, \;\;{\rm with}\;\;\; s = \frac{p_1-p}{p_1-p_0}s_0 +   \frac{p_0-p}{p_0-p_1}s_1.
\end{equation}  
In the case  $p_0 = +\infty$, we have
\begin{equation}\label{interpol2}
  \Vert u\Vert_{L^{p,s}(\R^d)} \leq  (\sup_{\R^d}\<x\>^{s_0}\vert u(x)\vert)^{1-p_1/p}
  (\Vert u\Vert_{L^{p_1,s_1}(\R^d)})^{p_1/p}, \;\;{\rm with}\;\; \;s = (p-p_1)s_0 + s_1.
\end{equation}

\subsection{Rough estimates of the harmonic oscillator} We recall here some more or less standard properties stated in \cite{ka1}. To begin with, we state a  "soft"  Sobolev inequality.
\begin{lemm} For all $u\in {\cal E}_H(I)$
\beq\label{sobg}
 \vert u(x)\vert \leq (\pi_H(I; x,x))^{1/2}\Vert u\Vert_{L^2(\R^d)}.
\eeq
\end{lemm}

\begin{proof}
We have 
$$
u(x) = \Pi u(x) = \int_{\R^d}\pi_H(I; x,y)u(y)\text{d}y.
$$
Using  the Cauchy-Schwarz inequality 
\beq\label{cs}
\vert u(x)\vert \leq \left(\int_{\R^d}\vert\pi_H(I; x,y)\vert^2\text{d}y\right)^{1/2}\Vert u\Vert_{L^2(\R^d)}.
\eeq
Now we use that  $\Pi_H(I)$ is an orthonormal projector.
\beq\label{proj}
\pi_H(I; x,y) = \int_{\R^d}\pi_H(I; x,z)\pi_H(I; z,y)\text{d}z \;\;{\rm and}\;\;  \pi_H(I; x,y) = \overline{\pi_H(I; y,x)}.
\eeq
Finally, from  (\ref{cs})  and  (\ref{proj})  with  $y=x$  we get  (\ref{sobg}).
\end{proof}

 The next result gives a bound on $\pi_{H}$.

\begin{lemm} The following bound holds true
\beq\label{karest}
\pi_H(\lambda; x,x) \leq C\lambda^{d/2}\exp\left(-c\frac{\vert x\vert^2}{\lambda}\right),\quad
\forall x\in\R^d,\lambda \geq 1.
\eeq
\end{lemm}

\begin{proof} 
Let  $K(t; x,y)$  be the heat kernel   of 
${\rm e}^{-t H}$.  It is given by the  following Mehler formula \footnote{The Mehler formula can also be obtained from the Fourier transform computation of the Weyl symbol of ${\rm e}^{-t H}$ (see \cite[Exercise IV-2]{ro}).}
\beq\label{mehl}
K(t; x,y)  = (2\pi\sinh 2t)^{-d/2}
\exp\left(-\frac{\tanh t}{4}\vert x+y\vert^2- \frac{ \vert x-y\vert^2}{4\tanh t}\right).
\eeq
So we have 
\beq\label{heat}
K(t; x,x) = \int_{\R}{\rm e}^{-t\mu}\,\text{d}\pi_{H}(\mu;x,x) = (2\pi\sinh 2t)^{-d/2}\exp(-\vert x\vert^2\tanh t).
\eeq
We set $t=\lambda^{-1}$, integrate in   $\mu$ on  $[0,\lambda]$ and  get 
$$
\pi_H(\lambda; x,x)\leq \e K(\lambda^{-1}; x,x).
$$
Assuming  $\lambda \geq \lambda_0$, $\lambda_0$ large enough, we  easily see that  (\ref{karest})
is a consequence of  (\ref{heat}). 
\end{proof}

Let  $u\in{\cal E}_H(\lambda)$.  From  (\ref{sobg}) and  (\ref{karest})  we get 
\begin{equation*}
\vert u(x)\vert \leq C\lambda^{d/4}\exp\left(-c\frac{\vert x\vert^2}{2\lambda}\right)\Vert u\Vert_{L^2(\R^d)},
\end{equation*}
where $c,C>0$ do not depend on  $x\in\R^d$ nor  $\lambda\geq 1$.

\begin{rema}
From \eqref{karest}, we can deduce that  for every $\theta>0$ there exists $C_\theta>0$ such that 
\begin{equation*}
\pi_H(\lambda; x,x) \leq C_\theta\lambda^{(d+\theta)/2}\<x\>^{-\theta} ,
\end{equation*}
     which  by \eqref{sobg} implies with the semiclassical parameter $h=\lambda^{-1}$
    \begin{equation*}
      \<x\>^{\theta/2}h^{(d+\theta)/4} \vert u(x)\vert \leq C_\theta\Vert u\Vert_{L^2(\R^d)}, \quad
\forall u\in {\cal E}_H(h^{-1}).
    \end{equation*} 
    We can easily see  that this uniform estimate  is true for  $u\in{\cal E}(I_h)$
where $I_h=[\frac{a}{h},  \frac{b}{h}]$ with $a < b$. For smaller  energy intervals we  can get much better 
estimates, as we will see in Lemma \ref{lemme5}.
 \end{rema}

\begin{rema} Let us compare the previous results with the case of a compact Riemannian  manifold~$M$,  and when $ H=-\triangle$ is the Laplace operator. We have the  uniform H\"ormander estimate~\cite{ho1}:
\beq\label{hor}
\pi_H(\lambda;x,x) =  c_d(x)\lambda^{d/2} + \mathcal{O}(\lambda^{(d-1)/2}),
\eeq
 where   $0<c_d(x)$ is a continuous function on $M$.  Thus from   (\ref{cs}) and~(\ref{hor})  we get for some constant~$C_S>0$, 
$$
\Vert u\Vert_{L^\infty(M)} \leq C_S\lambda^{d/4}\Vert u\Vert_{L^2(M)},\quad  \forall u\in{\cal E}(\lambda).
$$
\end{rema}

Let us emphasis here that  it results form the uniform Weyl law \eqref{hor}  that $\pi_H(\lambda;x,x)$
  has an upper bound and a lower bound of order $\lambda^{d/2}$. For confining potentials like $V$
  the behavior of $\pi_H(\lambda;x,x)$ is much more complicated because of the turning points: $\big\{\,\vert x\vert^2 = \lambda\,\big\}$.
 This behavior was analyzed in~\cite{ka1}.

\subsection{More refined  estimates for the spectral function } From  the Weyl law for the harmonic oscillator we have 
\begin{equation*} 
 N_H(\lambda) = c_d\lambda^d + \mathcal{O}(\lambda^{d-1}), \quad c_d>0,
 \end{equation*}
we deduce that   if (\ref{spect1}) is satisfied with $\delta =1$ then we have 
\begin{equation}\label{ordre}
\alpha h^{-d}(b_h - a_h) \leq N_h \leq \beta h^{-d}(b_h - a_h),\;\; \alpha>0, \beta>0.
\end{equation}
~

The main result of this section is the following lemma. It is a consequence of the work of Thangavelu \cite[Lemma 3.2.2, p. 70]{Thanga} on Hermite functions. This was proved later Karadzhov~\cite{ka1} with a different method. It could also be deduced from much more general results by Koch, Tataru and {Zworski~\cite{KOTA, KTZ}} and it is also related, after rescaling,  with results obtained by Ivrii \cite[Theorem 4.5.4]{iv}.
\begin{lemm}\ph \label{lemme5}
Let $d\geq 2$ and assume that $|\mu|\leq c_{0}$, $1\leq p\leq +\infty$ and $\theta \geq 0$. Then  there exists $C>0$ so that for all~$\lambda\geq 1$
\begin{equation*}
\|\pi_{H}(\lambda+\mu;x,x)-\pi_{H}(\lambda;x,x)\|_{L^{p,(p-1)\theta}(\R^d)} \leq C \lambda ^{  \a  },
\end{equation*}
with  $\a =\frac{d}2(1+\frac1p)-1 +\frac{\theta}{2}(1-\frac{1}{p})   $.
\end{lemm}

\begin{proof}
Recall the following estimates proved in  \cite[Theorem 4]{ka1}:   For $d\geq 2$ and $x\in \R$
\begin{equation}\label{kara1}
\vert\pi_H(\lambda +\mu;x,x) - \pi_H(\lambda ;x,x)\vert \leq C\lambda^{d/2 -1},\quad
\lambda \geq 1, \;\vert\mu\vert \leq 1.
\end{equation}
and for every $\varepsilon_0>0$ and every $N\geq 1$ there exists  $C_{\varepsilon_0, N}$ such that
\beq\label{kara2}
\pi_H(\lambda ;x,x) \leq C_{\eps_{0}, N}\vert x\vert^{-N},\; \;{\rm for}  \; \vert x\vert^2 \geq (1+\varepsilon_0)\lambda.
\eeq
 From  (\ref{kara1}) we get that for every $C_0>0$ there exists  $C>0$  such that
  \beq\label{kara3}
\vert\pi_H(\lambda +\mu;x,x) - \pi_H(\lambda ;x,x)\vert \leq C(1+\vert\mu\vert)\lambda^{d/2 -1},\quad
\lambda \geq 1, \vert\mu\vert \leq C_0\lambda.
\eeq
 Then from  \eqref{kara3} and \eqref{kara2} we get that for every $\theta \geq 0$ there exists $C$
such that 
  \beq\label{kara4}
\vert\pi_H(\lambda +\mu;x,x) - \pi_H(\lambda ;x,x)\vert \leq C(1+\vert\mu\vert)\lambda^{d/2 -1+\theta/2}
\<x\>^{-\theta},\quad
\lambda \geq 1, \vert\mu\vert \leq C_0\lambda.
\eeq
Therefore, by \eqref{kara2}, to get the result of Lemma \ref{lemme5}, it is enough to integrate the previous inequality on $\vert x\vert\leq c_0 \lambda^{1/2}$.
\end{proof}
From \eqref{kara4},  we easily get  an accurate estimate for the  spectral function 
$$
e_{x} = \pi_H(\frac{b_h}{h};x,x) -  \pi_H(\frac{a_h}{h};x,x).
$$
\begin{lemm} Assume that (\ref{spect1}) is satisfied with $0<\delta \leq 1$. 
For any  $\theta\geq 0$ there exists $C>0$ such that 
\beq\label{dsp}
\<x\>^\theta e_{x} \leq CN_hh^{(d-\theta)/2}.
\eeq
\end{lemm}
Using (\ref{sobg}) and interpolation inequalities we get  Sobolev type inequalities for $u\in{\cal E}_h$,
$\theta \geq 0$, $p\geq 2$.
\begin{equation}\label{detsob}
\Vert u\Vert_{L^{\infty,\theta/2}(\R^d)} \leq C  \left(N_hh^{(d-\theta)/2}\right)^{1/2}\Vert u\Vert_{L^{2}(\R^{d})},
\end{equation}
which in turn implies, by \eqref{interpol1}
\begin{equation}\label{detsob1r} 
 \Vert u\Vert_{L^{p,\theta(p/2-1)}(\R^d)} \leq C \left(N_hh^{(d-\theta)/2}\right)^{\frac{1}{2}-\frac{1}{p}}\Vert u\Vert_{L^{2}(\R^{d})}.
\end{equation}
By \eqref{ordre},  the previous inequality can be written as
\begin{equation*} 
 \Vert u\Vert_{L^{p,\theta(p/2-1)}(\R^d)} \leq C(b_h-a_h)^{\frac{1}{2}-\frac{1}{p}}
 h^{-(\frac{d+\theta}{2})(\frac{1}{2}-\frac{1}{p})}\Vert u\Vert_{L^{2}(\R^{d})}, \quad \forall p\in [2, +\infty], \;\forall\theta\in[0, d].
\end{equation*}

\begin{rema}
For similar bounds for eigenfunctions or quasimodes, we refer to  \cite{KTZ}.
\end{rema}
\section{Probabilistic weighted Sobolev estimates}\label{Sect4}
We apply here the general  probabilistic setting of Section \ref{Sect2} when  $K=H$ is the harmonic oscillator,
${\cal K} =L^2(\R^d)$ and $\{\varphi_j,\,j\in\N\}$ an orthonormal basis of Hermite  functions. Recall that  ${\bf S}_h $  is the unit sphere of the  complex
 Hilbert space ${\cal E}_h$, identified with $\C^{N}$ or $\R^{2N}$, and that   ${\bf P}_{\gamma,h}$ is  the probability on ${\bf S}_h$  defined as in Section \ref{Sect2}.\vspace{1\baselineskip}
 
 We divide  this section in two parts: in the first part, under Assumption \ref{Assum1}, we establish upper bounds and
 in the second part we obtain lower bounds, but only in the case of Gaussian random variables (Assumption \ref{Assum2}), and under the condition $0\leq \delta<2/3$.
 \subsection{Upper bounds}\label{UB}
 We suppose here that Assumption \ref{Assum1}, \eqref{condi1}  and  (\ref{spect1}) with $0\leq \delta\leq 1$ are  satisfied. Our result is the following 
 \begin{theo}\ph\label{Thm41} 
There exist $h_0\in]0, 1]$,   $c_2>0$ and $C>0$  such that if
 $c_1 = d(1 +d/4)$, 
   we have 
\beq\label{upbh}
{\bf  P}_{\gamma,h}\Big[u\in{\bf S}_h:\; h^{-\frac{d-\theta}{4}}\Vert u\Vert_{L^{\infty,\theta/2}(\R^d)} >\Lambda\Big] \leq Ch^{-c_1}{\rm e}^{-c_2\Lambda^2},\; 
 \forall \Lambda>0,\; \forall h\in]0, h_0].
\eeq 
\end{theo}


\begin{proof}
We adapt here the argument of \cite{bule}. To begin with, by  (\ref{dsp})  and (\ref{psph}), there exists  $c_{2} >0$ such that for every $\theta \in[0, d]$, every $x\in\R^d$, and every  $\Lambda>0$ we have 
\beq\label{proba30}
{\bf  P}_{\gamma,h}\Big[u\in{\bf S}_h:\; \<x\>^{\frac{\theta}{2}}h^{-\frac{d-\theta}{4}}\vert u(x)\vert >\Lambda\Big] \leq {\rm e}^{-c_2\Lambda^2}.
\eeq
Now, we will need a covering argument. Our configuration space is not compact but using (\ref{kara2}) we have, for every $u\in{\bf S}_h$, 
$$
\vert u(x)\vert \leq C_N\vert x\vert ^{-N},\;\; {\rm for}\,\; \vert x\vert \geq (1+\epsilon_0)h^{-1/2}.
$$
So  choosing  $R>0$ large enough it is sufficient to estimate  $u$  inside  the box 
$B_{R_h}=\{x\in\R^d,\; \vert x\vert_\infty \leq Rh^{-1/2}\}$.
We divide $B_{R_h}$ in small boxes of side with length $\tau$ small enough. We use
the gradient estimate
$$
\vert\nabla_xu(x)\vert \leq Ch^{-1/2-{d}/{4}},\quad \forall u\in{\bf S}_h,
$$
and (\ref{proba30}) at the center of each small box  to get the result.\\
For $x, x'\in\R^d$ we have
$$
\vert\<x\>^{\theta/2} u(x)-\<x'\>^{\theta/2} u(x') |\leq C(\<x\>^{\theta/2}\vert u(x)-u(x')\vert + \<x\>^{\theta/2}\vert x-x'\vert\vert u(x')\vert). 
$$
Let $\{Q_\tau\}_{\tau\in A}$ be a covering of $B_{R_h}$ with small  boxes $Q_\tau$ with center $x_\tau$ and side length $\tau$ small enough.\\
Then for every  $x\in Q_\tau$  we have 
\beq\label{app1}
h^{(\theta-d)/4}\vert\<x\>^{\theta/2} u(x) -\<x_\tau\>^{\theta/2} u(x_\tau)\vert \leq C\tau h^{-1/2 -d/4}.
\eeq
We choose 
\beq\label{app2}
\tau \thickapprox \frac{\epsilon\Lambda}{2C}h^{1/2+d/4}
\eeq
and $h_\epsilon>0$ such that
\beq\label{app3}
\vert x\vert_\infty >Rh^{-1/2} \Rightarrow h^{(\theta-d)/4}\<x\>^{\theta/2}\vert u(x)\vert 
\leq \frac{\epsilon\Lambda}{2},\quad \forall h\in]0, h_\epsilon].
\eeq
Then  using (\ref{proba30}),  (\ref{app1}), (\ref{app2}) and  (\ref{app3}) we get
\beq\label{app4}
{\bf  P}_{\gamma,h}\Big[u\in{\bf S}_h:\; h^{-\frac{d-\theta}{4}}\Vert u\Vert_{L^{\infty,\theta/2}(\R^d)} >\Lambda\Big]
\leq \#A{\rm e}^{-c_2(1-\epsilon)^2\Lambda^2},\quad  \forall \Lambda >0, \forall h\in ]0,h_\epsilon].
\eeq
Using now that $\#A \thickapprox Ch^{-c_1} $ with
$c_1 = d(1 +d/4)$ we get (\ref{upbh}) from (\ref{app4}).
\end{proof}

We can deduce probabilistic estimates for the derivatives as well. Recall that  the Sobolev spaces ${\cal W}^{s,p}(\R^d)$ are defined in (\ref{sobharm}).
\begin{coro}\ph \label{Coro42}
For any multi index $\alpha, \beta\in\N^d$ there exists $\tilde c_2$ such that 
$$
{\bf  P}_{\gamma,h}\Big[u\in{\bf S}_h:\; h^{\frac{\vert\alpha\vert +\vert\beta\vert}{2}-\frac{d}{4}}
\Vert x^\alpha\partial_x^\beta u\Vert_{L^{\infty}(\R^d)} >\Lambda\Big] \leq Ch^{-c_1}{\rm e}^{-\tilde c_2\Lambda^2},\quad 
 \forall \Lambda>0,\; \forall h\in]0, h_0].
$$
In particular  we have, for every $s>0$, 
$$
{\bf  P}_{\gamma,h}\Big[u\in{\bf S}_h:\; h^{\frac{s}{2}-\frac{d}{4}}
\Vert u\Vert_{\W^{s,\infty}(\R^d)} >\Lambda\Big] \leq Ch^{-c_1}{\rm e}^{-\tilde c_2\Lambda^2},\quad 
 \forall \Lambda>0,\; \forall h\in]0, h_0].
$$
\end{coro}
\begin{proof}
We apply (\ref{upbh}) using that from the spectral localization of $u\in{\cal E}_h$  we have
$$
\Vert x^\alpha\partial_x^\beta u\Vert_{L^2(\R^d)} \leq Ch^{-\frac{\vert\alpha\vert +\vert\beta\vert}{2}}\Vert u\Vert_{L^2(\R^d)},
$$ 
$$
\Vert H^su\Vert_{L^2(\R^d)} \leq Ch^{-s/2}\Vert u\Vert_{L^2(\R^d)}.
$$
\end{proof}

The following corollary shows that we  get a probabilistic Sobolev  estimate  improving the deterministic one
(\ref{detsob}) with probability close to one as $h\rightarrow 0$. The improvement is "almost" of order
$N_h^{1/2} \approx\left((b_h-a_h)h^{-d}\right)^{1/2}$. Choosing $\Lambda = \sqrt{-K\log h}$
for $K>0$  we get 
\begin{coro}
Let $c_{1},c_{2}>0$ be the constants given by Theorem \ref{Thm41}. Then for every $K>\frac{c_1}{c_2}$  we have
\begin{equation*}
{\bf  P}_{\gamma,h}\left[u\in{\bf S}_h:\; \Vert u\Vert_{L^{\infty,\theta/2}(\R^d)} >
 Kh^{\frac{d-\theta}{4}}\vert\log h\vert^{1/2}\right] \leq h^{Kc_2-c_1},\quad 
  \forall h\in]0, h_0],\;\forall\theta\in[0,d].
\end{equation*}
\begin{equation*}
{\bf  P}_{\gamma,h}\left[u\in{\bf S}_h:\; \Vert u\Vert_{{\cal W}^{s,\infty}(\R^d)} >
 Kh^{\frac{d}{4}-\frac{s}{2}}\vert\log h\vert^{1/2}\right] \leq h^{Kc_2-c_1},\quad 
  \forall h\in]0, h_0],\;\forall s\geq 0.
\end{equation*}
\end{coro}
Let us give now an application to a probabilistic Sobolev embedding for the Harmonic oscillator.\\
We shall use a Littlewood-Paley decomposition with $h_j =2^{-j}$. Let
$\theta$ a $C^\infty$ real function on $\R$ such that $\theta(t)=0$ for $t\leq a$,
$\theta(t)=1$ for $t\leq b/2$ with $0 <a < b/2$.  Define
$\psi_{-1}(t) = 1-\theta(t)$, $\psi_j(t) = \theta(h_jt) - \theta(h_{j+1}t)$ for $j\in\N$. Notice that the support of $\psi_j$ is in $[\frac{a}{h_j}, \frac{b}{h_j}]$.\\
For every  distribution $u\in{\cal S}'(\R^d)$ we have the Littlewood-Paley decomposition
\begin{equation*}
u = \sum_{j\geq -1}u_j,\quad {\rm  with}\quad u_j = \sum_{k\in\N}\psi_j(\lambda_k)\<u,\varphi_k\>\varphi_k
\end{equation*}
and we have  $u_j\in{\cal E}_{h_j}$. 

The Besov spaces for the Harmonic are naturally defined as follows: if $p, r\in[1,\infty]$ and $s\in\R$,   $u\in{\cal B}^s_{p,r}$ if and only if
$$
\Vert u\Vert_{\cal B^s_{p,r}} := \left(\sum_{j\geq-1}2^{jsr/2}\Vert u_j\Vert_{L^p(\R^d)}^r\right)^{1/r} <+\infty.
$$
We shall use here the spaces ${\cal B}^s_{2,\infty}$. For every $s>0$ we have
$$
 {\cal B}^s_{2,\infty} \subseteq L^2(\R^d) \subseteq  {\cal B}^0_{2,\infty} .
 $$
Another scale of spaces is defined as 
$$
{\cal G}^m = \big\{u\in \mathcal{S}'(\R^{d}):\; \sum_{j\geq 1}j^{m}\Vert u_j\Vert_{L^2(\R^d)} <+\infty\big\},\quad m\geq 0.
$$
Then for every $s>0$, $m\geq 0$ we have
$ {\cal B}^s_{2,\infty} \subseteq  {\cal G}^m  \subseteq L^2(\R^d)$.\\
It is not difficult to see that ${\cal G}^m$  can be compared with   the domain in $L^2(\R^d)$ 
of the operator ${\log^{s} H}$. This domain is  denoted by  ${\cal H}_{\log}^s$, the norm being the graph norm.
For every $s > 1/2$ we have
$$
{\cal H}_{\log}^{m+s} \subset {\cal G}^m \subset {\cal H}_{\log}^m.
$$
Notice that we do not  need that the energy localizations $\psi_j$ are smooth and we can define the same
spaces with $\psi(t) =\1_{[1, 2[}(t)$ so that the  energy intervals $[2^j, 2^{j+1}[$ are  disjoint.\ligne

Let us now define probabilities on ${\cal G}^m$  as we did for Sobolev  spaces ${\cal H}^s$.
Let $\gamma_j$ be a sequence  of complex numbers  satisfying (\ref{condi1}) and such that 
\beq\label{gm}
\sum_{j\geq0}j^{m}\vert\gamma\vert_{\Lambda_j} <+\infty,
\eeq
where $\Lambda_j = \Lambda_{h_j}$ and
$$
 v_\gamma^0= \sum_{j\geq 0} \gamma_j\varphi_j,\quad  v_\gamma(\omega) = \sum_{j\geq 0} \gamma_jX_j(\omega)\varphi_j,
$$
so that $v_\gamma$ is  a.s in ${\cal G}^m$   and its  probability law  defines a measure 
$\mu^m_\gamma$ in ${\cal G}^m$. This measure satisfies also the following properties as in Proposition
\ref{kakut}.

\begin{enumerate}[(i)]
\item  If the support of $\nu$ is $\R$ and if $\gamma_j\neq 0$ for all $j\geq 1$ then
the support of $\mu^m_\gamma$ is ${\cal G}^m$.
\item If $u^0_\gamma\in {\cal G}^m$ and $v_\gamma^0\notin {\cal G}^{s}$ 
where $s>m$ then  $\mu^m_\gamma( {\cal G}^s)=0$. In particular
$\mu_\gamma^m({\cal H}^s) = 0$ for every $s>0$.
\item Under the assumptions $\it(iii)$ in  Proposition \ref{kakut} we can construct singular measures $\mu^m_\gamma$ and $\mu^m_\beta$.
\end{enumerate}
\ligne

Now we can state the following corollary of Theorem \ref{Thm41}.
\begin{coro}\ph\label{glinft}
Suppose that $\gamma$ satisfies   (\ref{condi1}) with $a < b$ and  (\ref{gm})  with $m=1/2$.
Then for  the measure $\mu_\gamma^{1/2}$   almost all functions in the space  ${\cal G}^{1/2}$  are in the space
${\cal C}_H^{[d/2]}$ where 
 $$
{\cal C}_H^{\ell}(\R^d) = \Big\{u\in {\cal C}^\ell(\R^d):\; \Vert x^\alpha\partial_x^\beta u\Vert_{L^\infty(\R^d)}<+\infty,\;\;
\forall\,\vert\alpha\vert +\vert\beta\vert\leq \ell \Big\}.
$$
In particular   if $v_\gamma^0\in{\cal H}^{s_0}$, $s_0>0$ and if  $v_\gamma^0\notin{\cal H}^{s}$, $s>s_0$,
then we have $\mu_\gamma^{1/2}(\cal B^\sigma_{2,\infty})=1$ for every $\sigma>0$ and 
we have an  a.s embedding of the 
Besov space ${\cal B}^\sigma_{2,\infty}$  in ${\cal C}_H^{[d/2]}$. \end{coro}
\begin{proof}
Let $\di{u=\sum_{n\geq -1}u_n}\in{\cal G}^{1/2}$ with $u_n\in{\cal E}_{h_{n}}$.
For $\kappa >0$ (chosen  large  enough) denote by
$$
B^\kappa_n=\Big\{v\in {\cal E}_{h_{n}}:\; \Vert x^\alpha\partial_{x}^\beta v\Vert_{L^\infty(\R^d)} \leq 
\kappa\sqrt{ n}\Vert v\Vert_{L^2(\R^d)},\quad\forall \,\vert \alpha \vert+\vert \beta\vert\leq [d/2]\, \Big\}.
$$
We have, using Corollary \ref{Coro42}
$$
\nu_{\gamma, n}(B^\kappa_n)\geq  1- {\rm e}^{-n(c_2\kappa^2 - c_1)}.
$$
So if
$B^\kappa= \big\{u\in{\cal G}^{1/2}:\;u_0\in{\cal E}_{h_0},\;  u_n\in B^\kappa_n,\;\; \forall n\geq 1\big\}$, then we have
$$
\mu^{1/2}_\gamma(B^\kappa) \geq \prod_{n \geq 1}\big(1- {\rm e}^{-n(c_2\kappa^2 - c_1)}\big)\geq 1-\epsilon(\kappa)
$$
with $\di{\lim_{\kappa\rightarrow +\infty}\epsilon(\kappa) }=0$.  More precisely we have 
$\epsilon(\kappa) \approx {\rm e}^{-c\kappa^2}$ for some $c>0$.\\
Now if $u\in B^\kappa$    we have
\begin{equation*} 
\Vert x^\alpha\partial_{x}^\beta u\Vert_{L^\infty(\R^d)} \leq \sum_{n\geq-1} \Vert x^\alpha\partial_{x}^\beta u_n\Vert_{L^\infty(\R^d)} \leq \kappa\sum_{n\geq -1} \sqrt{ n}\Vert u_n\Vert_{L^2(\R^d)}:=\kappa\Vert u\Vert_{{\cal G}^{1/2}}.
\end{equation*}
So the corollary is proved.
\end{proof}
\begin{rema}
In the last corollary, for every $s>0$  we can choose $\gamma$ such that   $\mu^{1/2}_\gamma({\cal H^s})=0$.
So the smoothing property is  a probabilistic effect similar to the Khinchin inequality.\\
From the proof  we get a more quantitative  statement. 
There exists $c>0$ such that 
\begin{equation*}
\mu_\gamma^{1/2}\Big[\,\Vert u\Vert_{\W^{d/2,\infty}} \geq \kappa \Vert u\Vert_{{\cal G}^{1/2}} \Big] \leq {\rm e}^{-c\kappa^2}.
\end{equation*}
\end{rema}
\begin{rema}
The proof of the corollary depends on the  squeezing assumption (\ref{condi1}) on $\gamma$. For example
if  ~(\ref{condi1}) is satisfied for $b_h-a_h \approx h$ then we can consider the energy decomposition 
in intervals $[2n, 2(n+1)[$ instead of  the dyadic decomposition.
So when applying Theorem \ref{Thm41} with $h$ of order $\frac{1}{n}$  we get  $h^{-c_1}{\rm e}^{-c_2\Lambda^2} ={\rm e}^{c_1\log n-c_2\Lambda^2}$. 

Then taking $\Lambda = \kappa\sqrt{\log n}$ with 
$\kappa$ large enough, 
 in the construction of $B_n^\kappa$ we have to replace
$\sqrt n$ by $\sqrt{\log n}$. In the conclusion the space ${\cal G}^{1/2}$ is replaced by $\tilde{{\cal G}}^{1/2}$ where
$$
\di{\tilde{{\cal G}}^m = \big\{u\in \mathcal{S}'(\R^{d}):\; \sum_{j\geq 1}\log^m j\Vert u_j\Vert_{L^2(\R^d)} <+\infty\big\},\quad
u_j :=\sum_{2j\leq \lambda_n<2(j+1)}
\<u,\varphi_j\>\varphi_j}.
$$
\end{rema}
 
 
 \subsection{Lower bounds in the case of Gaussian random variables}\label{LB}
Here we suppose that the stronger Assumption \ref{Assum2}  and (\ref{spect1}) with $\delta < 2/3$
 are satisfied. We are interested to get a lower bound for $\Vert u\Vert_{L^{\infty,\theta/2}(\R^d)}$.\\
 The spectral condition $\delta <2/3$ is needed here because it seems difficult to estimate from below
 the variations of the spectral function of the harmonic oscillator in intervals of length $\leq h^{-1/3}$.\ligne

A first step is to get  two sides weighted  $L^r$ estimates  for large $r$ which is a    probabilistic  improvement  of (\ref{detsob1r}). Denote by 
\begin{equation}\label{def.b}
\beta_{r,\theta}=\frac{d-\theta}{2}(1-\frac2r).
\end{equation}
\begin{theo}\ph\label{2ESTr}
Assume  that $\theta\in[0,d]$, and denote by ${\cal M}_{r}$ a median of $ \Vert u\Vert_{L^{r, \theta(r/2-1)}}$. Then there exist $0<C_0 < C_1$, $K>0$,  $c_1>0$ , $h_0>0$ such that for all $r\in [2, K\vert\log h\vert]$ and  $h\in]0, h_0]$  such that 
\begin{equation}\label{2sidr}
{\bf P}_{\gamma,h}\Big[u\in {\bf S}_h : \Big| \Vert u\Vert_{L^{r, \theta(r/2-1)}}  - {\cal M}_{r}\Big| >\Lambda\Big]
\leq 2\exp\big(-c_2N_h^{2/r}h^{-\beta_{r,\theta}}\Lambda^2\big).
\end{equation}
and where
\begin{equation*}
 C_0\sqrt rh^{\frac{d-\theta}{4}(1-\frac{2}{r})} \leq {\cal M}_{r} \leq C_1\sqrt rh^{\frac{d-\theta}{4}(1-\frac{2}{r})},\quad
  \forall r\in[2, K\log N].
\end{equation*}
\end{theo}
This result shows that $ \Vert u\Vert_{L^{r, \theta(r/2-1)}}$ has a Gaussian concentration around its median.\ligne

From \eqref{2sidr} we deduce that  for every $\kappa\in]0, 1[$,  $K>0$, there exist $0<C_0 < C_1$,  $c_1>0$ , $h_0>0$ such that for all 
$r\in [2, K\vert\log h\vert^\kappa]$,  $h\in]0, h_0]$ and $\Lambda >0$  we have 
\begin{equation*}
{\bf P}_{\gamma,h}\left[u\in{\bf S}_h : C_0\sqrt r h^{\frac{d-\theta}{4}(1-\frac{2}{r})}\leq \Vert u\Vert_{L^{r, \theta(r/2-1)}}
 \leq C_1\sqrt rh^{\frac{d-\theta}{4}(1-\frac{2}{r})}\right] \geq 1- {\rm e}^{-c_1\vert\log h\vert^{1-\kappa}}.
\end{equation*}

As a consequence of Theorem \ref{2ESTr},  for every $\theta\in[0,d]$ we get a two sides weighted $L^\infty$ estimate showing that  Theorem \ref{Thm41} and its corollary are sharp.
\begin{coro}\ph\label{2ESTx}
After  a slight modification of the constants in Theorem \ref{2ESTr}, if necessary,  we get that for all  $\theta\in[0,d]$ and  $h\in]0, h_0]$
 \begin{equation}\label{2sidrX}
 {\bf P}_{\gamma,h}\left[u\in{\bf S}_h : C_0\vert\log h\vert^{1/2}h^{(d-\theta)/4} \leq \Vert u\Vert_{L^{\infty, \theta/2}} 
\leq C_1\vert\log h\vert^{1/2}h^{(d-\theta)/4} \right] \geq 1- h^{c_1}.
 \end{equation}
\end{coro}
 To prove these results we have to adapt  to the unbounded configuration space  $\R^d$ the proofs of 
 \cite[Theorems 4 and 5]{bule} which hold for  compact manifolds.  The  concentration result stated in Proposition~\ref{SC} will prove useful.
  
\begin{proof}[Proof of Theorem \ref{2ESTr}]
Denote by  $F_r(u) = \Vert u\Vert_{L^{r, \theta(r/2-1)}}$ and by ${\cal M}_{r}$ its median. Thanks to  (\ref{detsob1r}) we have the Lipschitz estimate 
$$  
\vert F_r(u)-F_r(v)\vert \leq C\left(N_hh^{\frac{d-\theta}{2}}\right)^{\frac{1}{2}-\frac{1}{r}}\|u-v\|_{L^{2}(\R^{d})},\quad \forall u, v\in{\bf S}_h.
$$
Therefore, by \eqref{supconcent} and \eqref{def.b}, we have for some $c_2>0$ 
\beq\label{medr}
{\bf P}_{\gamma,h}\Big[u\in {\bf S}_h : \vert F_{r}(u)  - {\cal M}_{r}\vert >\Lambda\Big]
\leq 2\exp\big(-c_2N_h^{2/r}h^{-\beta_{r,\theta}}\Lambda^2\big).
\eeq
The next step is to estimate ${\cal M}_{r}$. Denote by ${\cal A}_r^r = {\bf E}_h(F_r^r)$ the moment of  order $r$
  and compute, with $s=\theta(r/2-1)$, 
\begin{eqnarray*}
 {\cal A}_r^r  &=& {\bf E}_h\left(\int_{\R^d}\<x\>^s\vert u(x)\vert^r \,\text{d}x\right) \nonumber\\
 &=& r\int_{\R^d}\<x\>^{s}\Big
 (\int_0^{+\infty}s^{r-1}{\bf P}_{\gamma,h}\Big[  u\in {\bf S}_h : \vert u(x)\vert > s\Big]\,\text{d}s\Big)\,\text{d}x.
\end{eqnarray*}
Thus by \eqref{Binf} we  get 
 \begin{equation*} 
 C_1r \int_{\R^d}\<x\>^s 
 \Big(\int_0^{\epsilon_0\sqrt{e_x}}s^{r-1}{\rm e}^{-c_1\frac{N}{e_x}s^2}\,\text{d}s\Big)\text{d}x
 \leq {\cal A}_r^r  \leq   
  C_{2} r\int_{\R^d}\<x\>^s
 \Big(\int_0^{+\infty}s^{r-1}{\rm e}^{-c_{2}\frac{N}{e_x}s^2}\,\text{d}s\Big)\text{d}x.
\end{equation*}
Performing the change of variables  $ t =c_j \frac{ N}{e_{x}} s^{2}$ we obtain that  there  exist $C_{1},C_{2}>1$ such that                                   
\begin{equation}\label{sim}
C_{1}\,r (c_1 N)^{-r/2}\left(\int_{\R^d}\<x\>^se_x^{r/2}\text{d}x\right)  \int_0^{\eps N}t^{r/2-1}{\rm e}^{-t}\,\text{d}t      \leq    {\cal A}_r^r  \leq C_{2}\,r(c_2 N)^{-r/2}\left(\int_{\R^d}\<x\>^se_x^{r/2}\text{d}x\right) \Gamma(r/2),
\end{equation}
with $\eps=c_{1}\epsilon^{2}_0$. We need to estimate the term $\int_0^{\epsilon N}t^{r/2-1}{\rm e}^{-t}\text{d}t$ from below.
Using the elementary  estimate
$$
\int_{T}^{+\infty} t^{r/2-1}{\rm e}^{-t}\text{d}t \leq T^{r/2}{\rm e}^{1-T}\Gamma(r/2),\quad T\geq 1,
$$
we get that  there exists $\epsilon_1>0$ such that for $N$ large and 
$r\leq \epsilon_1\frac{N}{\log N}$ then we have
$$
\int_0^{\epsilon N}t^{r/2-1}{\rm e}^{-t}\text{d}t \geq \frac{\Gamma(r/2)}{2}.
$$
So we get the expected lower bound, $\forall  r\in[1,  \epsilon_1\frac{N}{\log N}]$, 
$$
  {\rm e}^{-r/2}C^{-1}r\left(\int_{\R^d}\<x\>^se_x^{r/2}\text{d}x\right)N^{-r/2}\Gamma(r/2)
  \leq  {\cal A}_r^r    \leq C_{2}\,rN^{-r/2}\left(\int_{\R^d}\<x\>^se_x^{r/2}\text{d}x\right) \Gamma(r/2).
$$
and where $\Gamma(r/2)$ can be estimated thanks to the Stirling formula: there exist $0<C_0 <C_1$  such that
$$
 (C_0r)^{r/2} \leq\Gamma(r/2) \leq (C_1r)^{r/2}, \quad \forall r\geq 1.
$$

Now we need  the following lemma   which will be proven in  Appendix \ref{AppB}. The upper bound can be seen as an application of Lemma \ref{lemme5} with $\lambda=h^{-1}$ and $\mu=(b_h-a_h)h^{-1}\sim N_hh^{d-1}$.
  \begin{lemm}\ph\label{2estspf}
Assume that $\theta>-d/(p-1)$. Then there exist $0<C_0 <C_1$ and $h_0>0$ such that
 \begin{equation*}
   C_0N_hh^{\beta_{2p,\theta}}\leq \left(\int_{\R^d}\<x\>^{\theta(p-1)}e^{p}_{x}\,\text{d}x\right)^{1/p}
  \leq   C_1N_hh^{\beta_{2p,\theta}}, 
 \end{equation*}
  for  every $p\in [1, \infty[$ and $h\in ]0, h_0]$ where $\beta_{r,\theta}=\frac{d-\theta}{2}(1-\frac{2}{r})$.
  \end{lemm}
  From this  lemma we get
  \beq\label{Aest}
  C_0\sqrt{rh^{\beta_{r,\theta}}} \leq {\cal A}_{r} \leq C_1\sqrt{rh^{\beta_{r,\theta}}},\quad \forall r\geq 2, \;h\in]0, h_0].
  \eeq
  Now we have to compare ${\cal A}_r$ and the median ${\cal M}_r$. We have
\begin{eqnarray*}
  \vert{\cal A}_r -{\cal M}_r\vert^r & = &\big\vert\Vert F_r\Vert_{L^r({\bf S}_h)} - 
  \Vert{\cal M}_r\Vert_{L^r({\bf S}_h)}\big\vert^r \nonumber\\
  & & \leq \Vert F_r  - {\cal M}_r\Vert^r_{L^r({\bf S}_h)}
   = r\int_0^\infty s^{r-1}
  {\bf P}_{\gamma,h}\big[\vert F_r - {\cal M}_r\vert >s\big]\text{d}s.
\end{eqnarray*}
  Then using the large deviation estimate (\ref{medr})   we get 
\begin{equation*}
  \vert{\cal A}_r -{\cal M}_r\vert \leq
  CN^{-1/r}\sqrt{rh^{\beta_{r,\theta}}},\quad \forall r\geq 2.
\end{equation*}
Choosing  $r\leq K\log N$, ($K<1$)  and $N$ large,    from (\ref{Aest}) we obtain
   \beq\label{Mrest}
  C_0\sqrt{rh^{\beta_{r,\theta}}} \leq {\cal M}_{r} \leq C_1\sqrt{rh^{\beta_{r,\theta}}},\quad
  \forall r\in[2, K\log N]
  \eeq
and     the proof of Theorem \ref{2ESTr}  follows using (\ref{Mrest})   and (\ref{medr}).
  \end{proof}  
  \begin{rema}
  The upper-bound in Lemma \ref{2estspf}  is true for $\delta =1$. This is proved in Appendix \ref{AppB}.
  \end{rema}
  
Now let us prove Corollary \ref{2ESTx}.
\begin{proof}[Proof of Corollary \ref{2ESTx}]
For simplicity we assume that $\theta=d$.
Using (\ref{upbh})  it is enough to prove that there exist $C_0>0$, $h_0>0$, $c_1>0$  such that
\beq\label{lpx}
{\bf P}_{\gamma,h}\left[u\in{\bf S}_h :  \Vert u\Vert_{L^{\infty, d/2}}\leq C_0\vert\log h\vert^{1/2}\right] \leq h^{c_1},\quad
\forall h\in]0, h_0].
\eeq
Let $u\in {\bf S}_h$, then by \eqref{interpol2} we have the interpolation inequality 
$$
\Vert u\Vert_{L^{r,d(r/2-1)}(\R^d)} \leq \Vert u\Vert_{L^{\infty, d/2}}^{1-2/r}.
$$ 
 So we get
 $$
 {\bf P}_{\gamma,h}\left[u\in{\bf S}_h :  \Vert u\Vert_{L^{\infty, d/2}}\leq C_0\vert\log h\vert^{1/2}\right]
 \leq {\bf P}_{\gamma,h}\left[u\in{\bf S}_h :  \Vert u\Vert_{L^{r, d(r/2-1)}}\leq \left(C_0\vert\log h\vert^{1/2}\right)^{1-2/r}\right],
 $$
 and choosing $r=r_h=\epsilon_0\vert\log h\vert$ we obtain
$$
  {\bf P}_{\gamma,h}\left[u\in{\bf S}_h :  \Vert u\Vert_{L^{\infty, d/2}}\leq C_0\vert\log h\vert^{1/2}\right] 
  \leq{\bf P}_{\gamma,h}\left[u\in{\bf S}_h :  \Vert u\Vert_{L^{r_{h}, d(r_{h}/2-1)}}\leq 
  \left(\frac{C_0}{\sqrt{\varepsilon_0}}r_h^{1/2}\right)^{1-2/r_h}\right].
$$ 
   Then choosing $h_0>0$, $\frac{C_0}{\sqrt{\varepsilon_0}}$ small enough and 
   $\Lambda = c\vert\log h\vert^{1/2}$   we can conclude 
   that (\ref{lpx}) is satisfied using  (\ref{2sidr}).
 \end{proof}  
 \begin{rema}
 Concerning the mean ${\cal M}_{\infty}$ of
  $F_\infty(u):=\Vert u\Vert_{L^{\infty,d/2}}$ it results from Corollary \ref{2ESTx},  (\ref{upbh}) 
  and (\ref{detsob}) that we have the two sides estimates
  $$
  C_0\vert\log h\vert^{1/2} \leq {\cal M}_{\infty}\leq C_1\vert\log h\vert^{1/2},\quad\forall h\in]0, h_0].
  $$ 
 \end{rema}
  It is not difficult to adapt  the  proof  of \eqref{2sidr} and  \eqref{2sidrX} for  the Sobolev norms
  $\Vert u\Vert_{{\cal W}^{s,p}(\R^d)}$. It is enough to remark that considering 
  $L_su(x) = H^{s/2}u(x)$ we have
  $$
  e_{L_s} = \sum_{j\in\Lambda}\lambda_j^s\varphi_j^2(x).
  $$
  But for $j\in\Lambda$, $\lambda_j$ is of order $h^{-1}$ hence there exists $C>0$ such that
  $$
  C^{-1}h^{-s}e_{x} \leq  e_{L_s} \leq Ch^{-s}e_{x}. 
  $$
  Using this property we easily  get the next result, which in particular implies Theorem \ref{thm1}.
  Let ${\cal M}_{r,s}$ be the median of $u\mapsto\Vert u\Vert_{ \cal W^{s,r}(\R^d)}$, and recall the definition \eqref{def.b}. Then 
  \begin{theo}\ph\label{thm412}
 Let $s\geq 0$. There exist $0<C_0 < C_1$, $K>0$,  $c_1>0$ , $h_0>0$ such that for all $r\in [2, K\vert\log h\vert]$ and  $h\in]0, h_0]$
 
 \begin{equation}\label{eq412}
{\bf P}_{\gamma,h}\Big[u\in {\bf S}_h : \Big| \Vert u\Vert_{\cal W^{s,r}(\R^d)}  - {\cal M}_{r,s}\Big| >\Lambda\Big]
\leq 2\exp\big(-c_2N_h^{2/r}h^{-\beta_{r,0}+s}\Lambda^2\big).
\end{equation}
where
\begin{equation*}
 C_0\sqrt rh^{\frac{\beta_{r,0}-s}{2}} \leq {\cal M}_{r,s} \leq C_1\sqrt rh^{\frac{\beta_{r,0}-s}{2}},\quad
  \forall r\in[2, K\log N].
  \end{equation*}
  In particular, for every $\kappa\in]0, 1[$, $K>0$ , there exist $C_0 >0$, $C_1>0$, $c_1>0$ such that 
  for  every $r\in [2, K\vert\log h\vert^\kappa]$ we have 
\begin{equation*}
{\bf P}_{\gamma,h}\left[u\in{\bf S}_h : C_0\sqrt r h^{\frac{d}{4}(1-\frac{2}{r})}h^{-\frac{s}2}\leq \Vert u\Vert_{{\cal W}^{s,r}(\R^d)}
 \leq C_1\sqrt r h^{\frac{d}{4}(1-\frac{2}{r})}h^{-\frac{s}2}\right] \geq 1-{\rm e}^{-c_1\vert\log h\vert^{1-\kappa}},
\end{equation*}
For $r=+\infty$ we have for all  $h\in]0, h_0]$
\begin{equation*}
{\bf P}_{\gamma,h}\left[u\in{\bf S}_h : C_0\vert\log h\vert^{1/2}h^{\frac{d-2s}{4}} \leq\Vert u\Vert_{{\cal W}^{s,\infty}(\R^d)}
\leq C_1\vert\log h\vert^{1/2}h^{\frac{d-2s}{4}} \right] \geq 1- h^{c_1}.
\end{equation*}
 \end{theo}
Namely,
 $$ \Vert u\Vert_{{\cal W}^{s,r}(\R^d)}\approx h^{-s/2} \Vert u\Vert_{{ L}^{r,0}(\R^d)} +  \Vert u\Vert_{{ L}^{r,s}(\R^d)}, $$
 and 
 $$h^{-s/2} \Vert u\Vert_{{ L}^{r,0}(\R^d)}  \sim h^{\frac{d}{4}(1-\frac{2}{r})}h^{-\frac{s}2} ,\quad  \Vert u\Vert_{{ L}^{r,s}(\R^d)}  \sim h^{\frac{d}{4}(1-\frac{2}{r})}h^{-\frac{s}{2r}}. $$
 
  \subsection{Lower bounds in the general case}\label{LB2}
Under  Assumption \ref{Assum1}, we  prove a weaker version of Theorem \ref{thm412}.
  \begin{theo}\ph\label{thm413}
Suppose that Assumption \ref{Assum1} is satisfied. Let $s\geq 0$, 
 $\kappa\in]0, 1[$, $K>0$.  There exist $0<C_0 < C_1$, $K>0$,  $c_1>0$, $h_0>0$ such that for all 
 $r\in [2, K\vert\log h\vert^\kappa]$ and  $h\in]0, h_0]$
\begin{equation*}
{\bf P}_{\gamma,h}\left[u\in{\bf S}_h : C_0 h^{\frac{d}{4}(1-\frac{2}{r})}h^{-\frac{s}2}\leq \Vert u\Vert_{{\cal W}^{s,r}(\R^d)}
 \leq C_1 \sqrt r  h^{\frac{d}{4}(1-\frac{2}{r})}h^{-\frac{s}2}\right] \geq 1- {\rm e}^{-c_1\vert\log h\vert^{1-\kappa}},.
\end{equation*}
For $r=+\infty$ we have for all  $h\in]0, h_0]$
\begin{equation*}
{\bf P}_{\gamma,h}\left[u\in{\bf S}_h : C_0h^{\frac{d-2s}{4}} \leq\Vert u\Vert_{{\cal W}^{s,\infty}(\R^d)}
\leq C_1\vert\log h\vert^{1/2} h^{\frac{d-2s}{4}} \right] \geq 1- h^{c_1}.
\end{equation*}
  \end{theo}
 Therefore, we have optimal constants in the control of the ${{\cal W}^{s,r}(\R^d)}$ norms when $r<+\infty$ and for general random variables which satisfy the concentration property, but when $r=+\infty$ we lose the factor $\vert\log h\vert^{1/2}$ in the lower bound.
 
 \begin{proof} We can  follow the main lines of the proof of Theorem \ref{thm412}. Here compared to \eqref{sim} we get 
 \begin{eqnarray*}
 {\cal A}_r^r   &\geq &C\,r N^{-r/2}\left(\int_{\R^d}\<x\>^se_x^{r/2}\text{d}x\right)  \int_0^{\eps }t^{r/2-1}{\rm e}^{-t}\,\text{d}t   \\
  &\geq &C N^{-r/2}\left(\int_{\R^d}\<x\>^se_x^{r/2}\text{d}x\right)  \eps^{r/2},
 \end{eqnarray*}
 and this explains the loss of the factor $\sqrt{r}$.
 \end{proof}

   \subsection{ Global probabilistic $L^p$-Sobolev estimates}    
        Here we extend the $L^\infty$- random estimates obtained before  to the $L^r$-spaces  for any real 
         $r\geq 2$, and we prove Theorem \ref{Stric}. Let us recall the definition \eqref{Besov} of the  Besov spaces, where we use    the notations of 
         Subsection \ref{UB} for the dyadic Littlewood-Paley decomposition.

         \begin{proof}[Proof of Theorem \ref{Stric}]
    Recall  that for every $\sigma >m$ we can choose $\gamma$
     such that $\mu_\gamma({\cal H}^\sigma) =0$.\\
     Denote by $F_{r,s}(u) =\Vert u\Vert_{{\cal W}^{s,r}}$. The Lipschitz norm of $F_{r,s}$
     satisfies
     $$
     \Vert F_{r,s}\Vert_{Lip} \leq Ch^{-s +d(\frac{1}{2} -\frac{1}{r})}N_h^{\frac{1}{2}-\frac{1}{r}}.
     $$
     Let us denote by ${\cal M}_{r,s}$  the median of $F_{r,s} $ on the sphere
     ${\bf S}_h$ for the probability ${\bf P}_{\gamma,h}$ and by  ${\cal A}_{r,s}^r$
     the mean of~$F_{r,s}^r$. From Proposition \ref{SC} we have, for some $0<c_0<c_1$, 
    \beq\label{ldr}
     {\bf P}_{\gamma,h}\Big[u\in{\bf S}_h: \vert F_{r,s}-{\cal M}_{r,s}\vert >K \Big]\leq
     \exp\Big(-c_1N\frac{K^2}{ \Vert F_{r,s}\Vert_{Lip}^2}\Big)\leq \exp\Big(-c_0N^{1/r}K^2\Big).
  \eeq
     With the same computations as for \eqref{Mrest} we get
     \beq\label{Apprs}
     {\cal A}_{r,s} \approx \sqrt r\quad {\rm and}\quad
     \vert {\cal A}_{r,s} - {\cal M}_{r,s}\vert \lessim\sqrt r N^{-1/r}.
     \eeq
     These formulas are obtained from (\ref{psph}) applied to the linear form
     $L_su := H^su(x)$ noticing that 
     $$
     e_{L_s} =\sum_{j\in\Lambda_h}\vert H^s\varphi_j(x)\vert^2 \approx h^{-2s}e_{x}.
     $$
     Then  taking $c_0>0$  small enough that we have
\beq\label{locrs}
\nu_{\gamma, h}\left[ \,v\in {\cal E}_h:\;  \Vert v\Vert_{{\cal W}^{s,r}} \geq  K\Vert v\Vert_{L^2(\R^d)}\,\right] \leq \exp\left(-c_0N^{2/r}K^2\right),\quad \forall K\geq 1.
\eeq
Then from (\ref{locrs}) we proceed as for the proof of Corollary \ref{glinft}. For simplicity we consider here the usual Littlewood-Paley decomposition.
Then  we have $N^{2/r} \approx 2^{2nd /r}$. So the end of the proof 
follows by considering 
$$
B^\kappa_n=\big\{v\in {\cal E}_n:\;  \Vert v\Vert_{{\cal W}^{s,r}} \leq K\Vert v\Vert_{L^2(\R^d)}\big\}.
$$
So for a fixed $r\geq 2$ we infer (\ref{lp0}) from (\ref{ldr}) and (\ref{Apprs}), taking $c_0>0$
small enough,   we get
$$
\mu_\gamma\left(\prod_{n\geq 0}B_n^K\right) \geq 1-{\rm e}^{-c_0K^2}.
$$
 \end{proof}
 
  Using the isometry  $u\mapsto H^{-m/2}u$
     between ${\cal B}_{2,1}^s$ and ${\cal B}_{2,1}^{m+s}$ for all real $m\geq 0$, we can get the following corollary to Theorem \ref{Stric}.
        \begin{coro}\label{coroBase} Let $m\geq 0$ and assume that $\gamma$ satisfies \eqref{condi1} and
        $$
         \sum_{n\geq 0}2^{nm}\vert\gamma\vert_n < +\infty.
         $$
    Then  for $s=d(\frac{1}{2} -\frac{1}{r}) + m$  and $r\geq 2$, we have
     \begin{equation*}
            \mu_{\gamma}\left[\,u\in {\cal B}_{2,1}^m :\; \Vert u\Vert_{{\cal W}^{m+s,r}} \geq  K\Vert u\Vert_{{\cal B}_{2,1}^m} \,\right]\leq {\rm e}^{-c_0K^2}.
     \end{equation*}
   \end{coro}

\section{Application to Hermite functions}\label{Sect5}


We turn to the proof of Theorem \ref{theoB} and we can follow the main lines of \cite[Section 3]{bule}. We use here the upper bounds estimates of Section \ref{UB} in their full strength. Firstly, we assume that for all $j\in \Lambda_{h}$, $\gamma_{j}=N_h^{-1/2}$ and that $X_{j}\sim \mathcal{N}_{\C}(0,1)$, so that ${\bf P}_{h}:={\bf P}_{\gamma,h}$ is the uniform probability on ${\bf S}_{h}$. We set $h_{k}=1/k$ with $k\in \N^{*}$, and  
 \begin{equation*}
 a_{h_{k}}=2+dh_{k},\quad b_{h_{k}}=2+(2+d)h_{k}.
 \end{equation*}
  Then   \eqref{spect1}  is satisfied with  $\delta=1$ and  $D=2$. In particular, each interval
   $$I_{h_{k}}=\Big[\,\frac{a_{h_{k}}}{h_{k}},\frac{b_{h_{k}}}{h_{k}}\Big[=[2k+d,2k+d+2[$$
    only contains  the  eigenvalue $\lambda_{k}=2k+d$ with multiplicity $N_{h_{k}}\sim ck^{d-1}$, and  $\mathcal{E}_{h_{k}}$ is the corresponding eigenspace of the harmonic oscillator $H$. We can identify the space of the orthonormal basis of $\mathcal{E}_{h_{k}}$ with the unitary group $U(N_{h_{k}})$ and  we endow $U(N_{h_{k}})$ with its Haar probability measure $\rho_{k}$.  Then the space $\mathcal{B}$ of the  Hilbertian bases of  eigenfunctions of $H$ in $L^{2}(\R^{d})$ can be identified with 
 \begin{equation*}
 \mathcal{B}=\times_{k\in \N} U(N_{h_{k}}),
 \end{equation*}
 which can be endowed with the measure 
 \begin{equation*}
\dis \text{d} \rho=\otimes_{k\in \N}\,\text{d} \rho_{k}.
 \end{equation*}
Denote by $B=(\phi_{k,\ell})_{k\in \N, \,\ell\in  \llbracket 1,N_{h_k}   \rrbracket } \in \mathcal{B}$ a typical orthonormal basis of $L^{2}(\R^{d})$ so that for all $k\in \N$, $(\phi_{k,\ell})_{\ell\in  \llbracket 1,N_{h_k} \rrbracket } \in U(N_{h_{k}})$ is an orthonormal basis of $\mathcal{E}_{h_{k}}$.\ligne

Then the main result of the section is the following, which implies Theorem \ref{theoB}.

\begin{theo}\label{theoBase}
Let $d\geq 2$. Then, if $M>0$ is large enough,   there exist $c,C>0$   so that for all $r>0$
 \begin{equation*}
         \rho\Big[\, B=(\phi_{k,\ell})_{k\in \N, \,
\ell\in  \llbracket 1,N_{h_k} \rrbracket } \in \mathcal{B}:   \exists k,\ell ; \; \Vert  \phi_{k,\ell} \Vert_{{\cal W}^{d/2,\infty}(\R^d)} \geq  M (\log k)^{1/2}+r  \,\Big]\leq C {\rm e}^{-c r^2}.
 \end{equation*}
 \end{theo}
 
 We will need the following result

 \begin{prop}
Let $d\geq 2$. Then, if $M>0$ is large enough,   there exist $c,C>0$   so that for all $r>0$ and $k\geq 1$
 \begin{multline}\label{51}
         \rho_{k}\Big[\, B_{k}=(\psi_{\ell})_{ 
\ell\in  \llbracket 1,N_{h_k} \rrbracket } \in U(N_{h_{k}}):   \exists \ell\in  \llbracket 1,N_{h_k} \rrbracket  ; \; \Vert  \psi_{\ell} \Vert_{{\cal W}^{d/2,\infty}(\R^d)} \geq  M (\log k)^{1/2}+r  \,\Big]\\
\leq C k^{-2}{\rm e}^{-c r^2}.
 \end{multline}
\end{prop} 

\begin{proof}
The proof is similar to the proof of \cite[Proposition 3.2]{bule}. We observe that for any ${ \ell_{0} \in \llbracket 1,N_{h_k} \rrbracket   } $, the measure $\rho_{k}$ is the image measure of ${\bf P}_{h_{k}}$ under the map 
$$U(N_{h_{k}}) \ni B_{k}=(\psi_{\ell})_{ 
\ell\in  \llbracket 1,N_{h_k} \rrbracket } \longmapsto \psi_{\ell_{0}} \in {\bf S}_{h_{k}}.$$
Then we use that ${\bf S}_{h_{k}}\subset \mathcal{E}_{h_{k}}$ is an eigenspace and by Theorem \ref{Thm41} we obtain that for all $\ell_{0}\in  \llbracket 1,N_{h_k} \rrbracket$
 \begin{multline*}
         \rho_{k}\Big[\, B_{k}=(\psi_{\ell})_{ 
\ell\in  \llbracket 1,N_{h_k} \rrbracket } \in U(N_{h_{k}}):      \; \Vert  \psi_{\ell_{0}} \Vert_{{\cal W}^{d/2,\infty}(\R^d)} \geq  M (\log k)^{1/2}+r  \,\Big]\\
\begin{aligned}
&= {\bf  P}_{h_{k}}\Big[u\in{\bf S}_{h_{k}}:\; \Vert u \Vert_{{\cal W}^{d/2,\infty}(\R^d)}  \geq M (\log k)^{1/2}+r\Big]\\
& = {\bf  P}_{h_{k}}\Big[u\in{\bf S}_{h_{k}}:\; k^{d/4}\Vert u \Vert_{{L}^{\infty,0}(\R^d)}  \geq M (\log k)^{1/2}+r\Big]  \\
& \leq C k^{c_{1}-M^{2}c_{2}}{\rm e}^{-c_{2}r^2},
\end{aligned}
 \end{multline*}
where $c_{1},c_{2}>0$ are given by Theorem \ref{Thm41}. As a consequence, \eqref{51} is bounded by $C k^{c}{\rm e}^{-c_{2}r^2}$, with $c=c_{1}-M^{2}c_{2}+d-1$ which implies the result.
  \end{proof}

\begin{proof}[Proof of Theorem \ref{theoBase}]
We set 
\begin{equation*}
\mathcal{F}_{k,r}=\big\{   \, B_{k}=(\psi_{\ell})_{ 
\ell\in  \llbracket 1,N_{h_k} \rrbracket } \in U(N_{h_{k}}):      \forall \ell \in  \llbracket 1,N_{h_k} \rrbracket  ; \;   \; \Vert  \psi_{\ell } \Vert_{{\cal W}^{d/2,\infty}(\R^d)} \leq   M (\log k)^{1/2}+r   \big\},
\end{equation*}
and $\dis \mathcal{F}_{r}=\cap_{k\geq 1} \mathcal{F}_{k,r}$. Then for all $r>0$\
\begin{equation*}
\rho(\mathcal{F}^{c}_{r})\leq \sum_{k\geq1}\rho_{k}(\mathcal{F}^{c}_{k,r})\leq C \sum_{k\geq1} k^{-2}{\rm e}^{-c r^2}= C'\e^{-c r^{2}},
\end{equation*}
 and this completes the proof.
  \end{proof}     
  
  We have the following consequence of the previous results.
  \begin{coro}
  For $\rho$-almost all orthonormal basis $(\phi_{k,\ell})_{k\in \N, \,\ell\in  \llbracket 1,N_{h_k} \rrbracket }$
  of eigenfunctions of~$H$  we have
  $$
  \Vert\varphi_{k,\ell}\Vert_{L^\infty(\R^d)} \leq (M+1)k^{-\frac{d}{4}}(1+\log k)^{1/2},\qquad \forall\,k\in \N, \;\forall \,\ell\in\llbracket 1,N_{h_k} \rrbracket.
  $$
    \end{coro} 
    \begin{proof}
    Apply (\ref{51}) with $r=(\log k)^{1/2}$ and denote,  for $k\geq 2$, $\Omega_k$ the event
    $$
    \Omega_k = \big\{B=(\varphi_{k,\ell}),\;\exists\ell\in\llbracket 1,N_{h_k} \rrbracket,\;  \Vert\varphi_{k,\ell}\Vert_{L^\infty(\R^d)}\geq (M+1)k^{-d/4}
    (\log k)^{1/2}\big\}.
    $$
    We have $\rho(\Omega_k) \leq \frac{C}{k^2}$.  Therefore from the Borel-Cantelli Lemma  we have
    $\di{\rho[\limsup \Omega_k] }= 0$ and this gives the corollary.
    \end{proof}
   \appendix
   \section{Proof of Proposition \texorpdfstring{\ref{kakut} $(iii)$}{2.4} }\label{AppA}
   \begin{proof}
   Denote by $f_\gamma(x) = \frac{c_{\alpha}}{\gamma}\e^{-(\frac{|x|}{\gamma})^{\alpha}}$,\; $\gamma>0$. We have,
   with obvious identifications, 
   $$
   \di{\mu_\gamma = \otimes_{j\geq 0}(f_{\gamma_j}\text{d}x)}.
   $$
   Denote by
   $$
   \pi_j = \int_\R\left(\frac{f_{\gamma_j}}{f_{\beta_j}}\right)^{1/2}f_{\beta_j}\text{d}x.
   $$
   According to the main result of \cite{kaku} the measures $\mu_\gamma$ and $\mu_{\beta}$ are mutually singular if the infinite product $\prod_{j\geq 0}\pi_{j}$ is divergent.  From elementary computations we get
   $$
   \pi_j = \left(\frac{1}{2}\left(\frac{\gamma_j}{\beta_j}\right)^{\alpha/2} + \frac{1}{2}
   \left(\frac{\beta_j}{\gamma_j}\right)^{\alpha/2}\right)^{-1/\alpha}.
   $$
$\bullet$   If $\pi_j$  has not 1 as limit then the product is divergent.\\
 $\bullet$   If $\pi_j$  has 1 as limit then  the infinite product is divergent if
   $\di{\sum_{j\geq 0}(\pi_j^{-\alpha}-1) =+\infty}$. So, 
    using that
   $$
   \frac{1}{2}(x+\frac{1}{x}) = 1+\frac12(1-x)^2 + {\cal O}(1-x)^3,
   $$
   we see that the infinite product is divergent if (\ref{critka}) is satisfied.
   \end{proof}
   
     \section{\texorpdfstring{$L^p$}{Lp}\label{AppB}  weighted   spectral estimates for the Harmonic oscillator}
     Our goal here is to give a self-contained proof of Lemma \ref{2estspf}.
     It could be proved using the semi-classical functional calculus for pseudo-differential operators \cite{ro}, but for the harmonic oscillator it is possible to use the exact Mehler formula and elementary properties
     of  Hermite functions to get the result.

 \subsection{A functional calculus with parameter for the Harmonic oscillator}
 The starting point is the inverse Fourier transform
 $$
 f(H) = \frac{1}{2\pi}\int {\rm e}^{itH}\hat f(t)\text{d}t,
 $$
 where $f$ is in the Schwartz space ${\cal S}(\R)$.
 
  We want estimates for the integral kernel 
 $K_f(x,y)$ of $f(H)$. To do that it is convenient to  first compute the Weyl symbol $W_{f(H)}(x,\xi)$ of 
 $f(H)$ and use that
 $$
 K_f(x,y)  = (2\pi)^{-d}\int_{\R^d}W_{f(H)}\big(\frac{x+y}{2},\xi\big){\rm e}^{i(x-y)\cdot\xi}\text{d}\xi.
 $$
 For basic properties about the Weyl calculus see for example \cite{ro}.   The  unitary operator ${\rm e}^{itH}$ has an explicit Weyl symbol $w(t,x,\xi)$ :
 \beq\label{mw}
 w(t,x,\xi) = \frac{1}{(\cos t)^d}{\rm e}^{i\tan t(\vert x\vert^2+\vert\xi\vert^2)},\quad {\rm for}\;\;\vert t\vert <\frac{\pi}{2}.
 \eeq
 Formula (\ref{mw}) can be easily  proved  from the Mehler formula \eqref{mehl} and also directly (see \cite[Exercise~IV]{ro}).
 
 Let us introduce a cutoff $\chi\in C^\infty(\R)$, $\chi(t) = 1$ for $\vert t\vert <\epsilon_0$,
 $\chi(t) = 0$ for $\vert t\vert > 2\epsilon_0$  with $0<\epsilon_0<\pi/4$. Denote by
 $$
 R_f = \frac{1}{2\pi }\int {\rm e}^{itH}(1-\chi(t))\hat f(t)\text{d}t
 $$
 and 
 \begin{equation*}
  \tilde{W}_f(x,\xi) = \frac{1}{2\pi}\int_{\R}w(t,x,\xi)\chi(t)\hat f(t)\text{d}t.
 \end{equation*}
 We apply these formulas to give estimates with  $f_h(s) = f(hs)$ where $h>0$ is a small parameter.
 We begin with an estimate  for the remainder term for the kernel $KR_{f_h}(x,y)$ of the operator
 $R_{f_h}$. 
 \begin{lemm}
 There exists $M_0>0$ such that for every $M\geq 1$ there exists $C_M>0$  such that
 \beq\label{rf}
 \vert KR_{f_h}(x,y)\vert \leq C_Mh^M\big(\<x\>\<y\>\big)^{-\frac{M}{2}-d-1}\Vert f\Vert_{M+M_0},\quad \forall h\in]0, 1],\; \forall x, y\in\R^d,
 \eeq
 where $\di{\Vert f\Vert_{m} = \sup_{j+k\leq m, t\in\R}\vert t^j\frac{d^k}{\text{d}t^k}\hat f(t)\vert}$.

 \end{lemm}
 \begin{proof}
 Denote by
 $\hat g_h(t) = (1-\chi(t))\hat f(\frac{t}{h})$.  So we have $R_{f,h} = g_h(H)$ and for every $M, M'\geq 1$,

 \begin{equation*}
  \vert\mu^Mg_h(\mu)\vert \leq \int_{\vert t\vert\geq\epsilon_0}\big\vert\frac{\text{d}^M}{\text{d}t^M}\hat g_h(t)\big\vert \text{d}t \leq C_{M,M'}(f)h^{M'}.
 \end{equation*}
 So we have
 $$
 \vert K R_{f_h}(x,y)\vert =\big\vert\sum_{j}g_h(\lambda_j)\varphi_j(x)\overline{\varphi_j(y)}\big\vert \leq 
 C h^{M}\big(\sum_{j}\lambda_j^{-M}\vert\varphi_j(x)\vert^2\big)^{1/2}\big(\sum_{j}\lambda_j^{-M}\vert\varphi_j(y)\vert^2\big)^{1/2}.
 $$
  Recall the Sobolev estimate in the harmonic spaces: for every $s>\frac{d}{2} +r$ there exists $C=C_{sr}$  such that
  $$
  \<x\>^r\vert u(x)\vert \leq C\Vert u\Vert_{{\cal H}^s},\quad \forall u\in{\cal H}^s(\R^d).
  $$
  So we get, for $s>\frac{d}{2} +r$, 
  $$
  \vert K R_{f_h}(x,y)\vert \leq Ch^M(\<x\>\<y\>)^{-r}\sum_j\lambda_j^{s-M}.
  $$
  Using that $\lambda_j\approx j^{1/d}$ and choosing $r = \frac{M}{2}+d+1$ we get (\ref{rf}).
  \end{proof}
  
  Our aim is to estimate  the kernel  of   $f\left(\frac{H-\nu\lambda}{\mu}\right)$
   for large $\lambda$,  $\vert\mu\vert\geq D\lambda^{1-\delta}$ where $D>0$ and $\delta  < 2/3$. 
The parameter   $\nu$ is   fixed in an interval $[\nu_0, \nu_1]$, 
    where $0<\nu_0 <\nu_1$.  All our estimates will be uniform in $\nu$, so for convenience we shall take
    $\nu=1$.
    
   Denote by $g_{\lambda,\mu}(s) = f\left(\frac{s-\lambda}{\mu}\right)$ so we have
   $\hat  g_{\lambda,\mu}(t) = \mu{\rm e}^{-it\lambda}\hat f(\mu t)$.
   We consider the dilated Weyl symbol: 
    $W_{\lambda,\mu}(x,\xi) = \tilde W_{g_{\lambda,\mu}}(\sqrt\lambda x, \sqrt\lambda\xi)$. 
   Then  we have
\beq\label{wes}
   W_{\lambda,\mu}(x,\xi)  = \frac{\mu}{2\pi}\int_{\R}
   {\rm e}^{i\lambda \Phi(t,x,\xi)}
  \frac{\chi(t)}{\cos(t)^d}\hat f(\mu t)\text{d}t,
 \eeq
   with the phase  
  $ \Phi(t,x,\xi) = \tan t (\vert x\vert^2+\vert\xi\vert^2) -t$.

\begin{lemm} Assume that $\delta < \frac{2}{3}$.   Then for every  $N, M\geq 0$ we have
\bea\label{cafun}
 W_{\lambda,\mu}(x,\xi)  = &\di{ \sum_{j(1-\delta) + k(2-3\delta)<N }}
 c_{k,j}\lambda^k\mu^{-3k-j}(\vert x\vert^2 + \vert\xi\vert^2)^k
 f^{(3k+j)}\Big(\frac{\lambda}{\mu}(\vert x\vert^2 + \vert\xi\vert^2-1)\Big) \nonumber\\
    &  + \;\;  {\cal O}\big(\lambda^{-N}(1+ \vert x\vert^2 + \vert\xi\vert^2)^{-M}\big),
  \eea
where $c_{k,j}$ are real  numbers, $c_{0,0}=1$.
\end{lemm}

\begin{proof}
  Using that $\partial_t \Phi(t,x,\xi) = (1 + \tan^2 t)(\vert x\vert^2+\vert\xi\vert^2) -1$
  and integrating by parts we get that for every $M$ there exists $C_M>0$ such that
  for $\left\vert\vert x\vert^2+\vert\xi\vert^2 -1\right\vert \geq 1/2$ we have
  \begin{equation}\label{decay}
    \vert W_{\lambda,\mu}(x,\xi) \vert \leq C_M \lambda^{-M\delta}(\vert x\vert^2+\vert\xi\vert^2+1)^{-M}.
  \end{equation}
  So it is enough to estimate  $ W_{\lambda,\mu}(x,\xi) $ for 
$\vert x\vert^2+\vert\xi\vert^2 \approx 1$.\\
To do that, we write down
$$
{\rm e}^{i\lambda\Phi(t,x,\xi)} = {\rm e}^{i\lambda(\tan t -t)(\vert x\vert^2 + \vert\xi\vert^2)}
{\rm e}^{i\lambda t(\vert x\vert^2 + \vert\xi\vert^2 -1)}.
$$
Denote by $E_t=(\tan t -t)(\vert x\vert^2 + \vert\xi\vert^2)$, then we have
$$
{\rm e}^{i\lambda E_t} = \sum_{0\leq k\leq N} \frac{(i\lambda E_t)^k}{k!} + r_N(i\lambda E_t)
$$
where
$$
\vert\frac{d^j}{ds^j}r_N(s)\vert \leq \frac{\vert s\vert^{N-j}}{(N-j)!},\quad 0\leq j\leq N.
$$
Lastly,  we end up the computation by expanding $(\tan t - t)$ with Taylor
\begin{equation*}
(\tan t-t )^{k}  \frac{\chi(t)}{\cos(t)^d}=\sum_{j=0}^{+\infty}d_{k,j}t^{3k+j}.
\end{equation*}
Thus
\begin{eqnarray*}
W_{\lambda,\mu}(x,\xi)&=& \frac{\mu}{2\pi} \sum_{k=0}^{+\infty}\sum_{j=0}^{+\infty}   \int_{\R}
   {\rm e}^{i\lambda t(\vert x\vert^2 + \vert\xi\vert^2 -1)} d_{k,j}i^{k}\lambda^{k}t^{3k+j}(|x|^{2}+|\xi|^{2})^{k}
  \hat f(\mu t)\text{d}t \nonumber \\
&=&  \frac{\mu}{2\pi} \sum_{k=0}^{+\infty}\sum_{j=0}^{+\infty}   \int_{\R}
   {\rm e}^{i\lambda t(\vert x\vert^2 + \vert\xi\vert^2 -1)} d_{k,j}i^{j}\lambda^{k}\mu^{-3k-j}(|x|^{2}+|\xi|^{2})^{k}
  \widehat{ f^{(3k+j)}}(\mu t)\text{d}t  \nonumber \\
  &=&   \sum_{k=0}^{+\infty}\sum_{j=0}^{+\infty}   
    c_{k,j}\lambda^{k}\mu^{-3k-j}(|x|^{2}+|\xi|^{2})^{k}
f \Big(\frac{\lambda}{\mu}(\vert x\vert^2 + \vert\xi\vert^2-1)\Big), 
\end{eqnarray*}
which implies the result with \eqref{decay}.
\end{proof}

\subsection {Proof of  Lemma \ref{2estspf}} 
    First remark 
  that when $\theta\geq 0$, the upper-bound  is a direct consequence of (\ref{kara2}) and (\ref{kara3})
   and this holds true for $\delta =1$. 
 The bound  (\ref{kara3}) being  a rather difficult result, we shall  prove by the same method the estimate
   from above and from below for $\delta< 2/3$.
   
  We   use here the   functional calculus  with energy parameter (\ref{cafun}).
  Let $f$ be a  non negative $C^\infty$ function in $]-2C_0, 2C_0[$ with  a compact support,
    such that $f=1$ in $[-C_0,  C_0]$.  
         We choose two cutoff functions $f_\pm$ with $f_+$ as above and $f_-$ such
     that supp$(f_-)\subseteq]C_1, C_0[$, $f_-=1$ in $[2C_1, C_0/2]$ where $C_1 < C_0/4$.
     If $K_{\pm,h}(x,y)$ is the Schwartz kernel  of 
     $f_\pm\left(\frac{H- h^{-1}}{\mu}\right)$ ($h =\frac{1}{\lambda}$ is now a small parameter, 
     $\nu\in[\nu_0,\nu_1]$). 
     We have
\begin{equation*}
     K_{-,h}(x,x) \leq e_{x} \leq K_{+,h}(x,x).
\end{equation*}
     So we have to prove
     \beq\label{reduc}
  C_0N_hh^{\beta_{2p,\theta}}\leq \left(\int_{\R^d}\<x\>^{\theta(p-1)}K_{\pm,h}(x,x)^p\text{d}x\right)^{1/p}
  \leq   C_1N_hh^{\beta_{2p,\theta}}.
  \eeq
  Recall that 
  $$
  K_{\pm,h}(x,x) = (2\pi)^{-d}\int_{\R^d}W_{\pm,h}(x,\xi)\text{d}\xi
 $$
  where $W_{\pm,h}(x,\xi)$ is the Weyl symbol of the operator
  $f_\pm\left(\frac{H-h^{-1}}{\mu}\right)$. So using (\ref{cafun})
  it is not difficult to see that it is enough to consider only the principal term  given by the following formula
\begin{equation*}
   K^0_{f,h}(x,x) = (2\pi )^{-d}\int_{\R^d} f\left(\frac{\vert x\vert^2 + \vert\xi\vert^2- h^{-1}}{\mu}\right)\text{d}\xi.
\end{equation*}

   We  shall detail now the lower-bound; the upper-bound is proved in the same  way. Denote by $K^0_-(x) = K^0_{f,h}(x,x) $ and $s= \theta(p-1)$. We have, with the change of variable $x=h^{-1/2}y$, $\xi=h^{-1/2}\eta$
   \begin{eqnarray*}
    \int_{\R^d}\<x\>^sK^0_-(x)^p\text{d}x &=& (2\pi)^{-dp}\int_{\R_x^d}\<x\>^s
   \left(\int_{\R_\xi^d}f_-\left(\frac{\vert x\vert^2+\vert\xi\vert^2-h^{-1}}{\mu}\right)\text{d}\xi\right)^p\text{d}x\\
  &=&  (2\pi)^{-dp}h^{-(1+p)d/2}\int_{\R_y^d}\<h^{-1/2}y\>^s
   \left(\int_{\R_\eta^d}f_-\left(\frac{\vert y\vert^2+\vert\eta\vert^2-1}{h\mu}\right)\text{d}\eta\right)^p\text{d}y
   \end{eqnarray*}
Using the  property of the support
 of $f_-$   we obtain
 $$
\int_{\R_\eta^d}f_-\left(\frac{\vert y\vert^2+\vert\eta\vert^2-1}{h\mu}\right)\text{d}\eta \gtrsim h\mu,
 $$
and that $|y|\leq 1$ on the support of $f_{-}$. Next,
\begin{equation*}
\int_{|y|\leq 1}\<h^{-1/2}y\>^s\text{d}y=h^{d/2}\int_{|x|\leq h^{-1/2}}\<x\>^s\text{d}x\sim  \left\{\begin{array}{ll} 
Ch^{d/2},\quad &\text{if} \quad s<-d, \\[6pt]  
C|\ln h|h^{d/2},\quad &\text{if} \quad s=-d, \\[6pt]  
Ch^{-s/2},\quad &\text{if} \quad s>-d.
\end{array} \right.
\end{equation*}
 Finally, we get (\ref{reduc})  using that $\mu\approx h^{\delta-1}$ so  $\mu h\approx h^\delta \approx h^dN_h$.

  \section{Proof of  \texorpdfstring{\eqref{phi}}{(2.12)}}\label{AppC}

 To begin with, we identify the complex sphere of $\C^{N}$ with the real sphere 
 $$   \mathbb{S}^{2N-1}=\big\{ w\in \R^{2N}\;:\; w^{2}_{1}+\dots+w^{2}_{2N}=1      \big\} \subset \R^{2N}.   $$
  Denote by ${\P}_{N}$ the uniform probability measure on $\mathbb{S}^{2N-1}$ and by $\mu_{N}$ the Gaussian measure on $\R^{2N}$ of density $\dis \text{d}\mu= \frac1{(2\pi)^{N}}  \exp\big(-\frac12\sum^{2N}_{j=1}x^{2}_{j}\,\big)\text{d}x_{1}\dots \text{d}x_{2N}$. It is easy to check that ${\P}_{N}$ is the image measure of $\mu_{N}$ by the map 
 \begin{equation*}
 \begin{array}{rcl}
G:\R^{2N}&\longrightarrow&\mathbb{S}^{2N-1}\\[3pt] 
\dis (x_{1},\dots,x_{2N})&\longmapsto &\dis \frac{1}{\sqrt{\sum^{2N}_{j=1}x^{2}_{j}}}(x_{1},\dots,x_{2N}).
 \end{array}
 \end{equation*}
Indeed, $\mu_{N}\circ G^{-1}$ is a probability measure on $\mathbb{S}^{2N-1}$ which is  invariant by the isometries of $\mathbb{S}^{2N-1}$, therefore ${\P}_{N}=\mu_{N}\circ G^{-1}$. For $t\in [0,1]$, denote by $\Phi(t)={\P}_{N}\big(\sqrt{w^{2}_{1}+w^{2}_{2}}>t\big)$, then 
\begin{eqnarray*}
\Phi(t)&=&\frac{1}{(2\pi)^{N}}\int \1_{x^{2}_{1}+x^{2}_{2}>t^{2}\sum_{j=1}^{2N}x^{2}_{j}}\,\e^{-\frac12\sum^{2N}_{j=1}x^{2}_{j}}\,\text{d}x_{1}\dots \text{d}x_{2N}\\
&=&\frac{1}{(2\pi)^{N}}\int \1_{x^{2}_{1}+x^{2}_{2}>\frac{t^{2}}{1-t^{2}}\sum_{j=3}^{2N}x^{2}_{j}}\,\e^{-\frac12\sum^{2N}_{j=1}x^{2}_{j}}\,\text{d}x_{1}\dots \text{d}x_{2N}.
\end{eqnarray*}
 We make a spherical change of variables $(x_{3},\dots, x_{2N})\longmapsto r\sigma$ and the polar change of variables $(x_{1},x_{2})=(r\cos \theta,r\sin\theta)$. Denote by $s=t/\sqrt{1-t^{2}}$, thus there exists $C_{N}$ so that 
 \begin{eqnarray*}
\Phi(t)&=&C_{N}\int r^{2N-3}\e^{-\frac12(\rho^{2}+r^{2})}\1_{\rho >sr}\,\rho \text{d}r\text{d}\rho\\
&=&C_{N}\int_{0}^{+\infty} r^{2N-3}\e^{-\frac12(1+s^{2})r^{2}} \, \text{d}r. 
\end{eqnarray*}
 Now, by the  change of variables $r'=(1+s^{2})^{1/2}r$, there exists $C_{N}$ so that
 \begin{equation*}
\Phi(t)=C_{N}(1+s^{2})^{-(N-1)}=C_{N}(1-t^{2})^{N-1},
 \end{equation*}
 and  $C_{N}=\Phi(0)=1$.

\section{Proof of Proposition \ref{SC}}\label{AppD}
 For simplicity we assume that the random variables, the $\gamma_j$  and the space ${\cal E}_h$  are real, and we identify~${\cal E}_h$  with $\R^N$, endowed  with its natural Euclidean norm~$\vert y\vert_0$. We   also  consider the $\gamma$-dependent norm
  $$
 \vert y\vert_\gamma^2 = \frac{1}{N}\sum_{1\leq j\leq N}\frac{y_j^2}{\gamma_j^2},\quad y=(y_1,\cdots, y_N).
 $$
Condition  \eqref{condi0} means that we have
 \begin{equation*}
  \frac{1}{C}\vert y\vert_0^2 \leq \vert y\vert_\gamma^2 \leq C\vert y\vert_0^2.
 \end{equation*}
We define a  probability measure $\nu_{\gamma}$ in $\R^N$ as the pull forward   of  the   measure $\nu$
 in $\R^N$  by the mapping 
   $\varphi$: $(x_1,\cdots, x_N)\mapsto \sqrt N(\gamma_1x_1,\cdots, \gamma_Nx_N)$. Notice that $\nu_{\gamma}$ satisfies the concentration property of Definition~\ref{defiCM}.

 Now we follow the proof of  of \cite[Proposition 2.10]{led}.  Let  $F$ be a  Lipschitz function on the sphere~${\bf S}_h$, and by homogeneity, it is enough to assume that $F$ is 1-Lipschitz. For $x\in\R^N$, 
$G(x) = \vert x\vert_0\big(F(\frac{x}{\vert x\vert_{0}})-F(x_0)\big)$, where $x_0$ is a fixed point on ${\bf S}_h$. Thus $G$  satisfies
\begin{equation}\label{lip}
\vert G(x) -G(y)\vert_0 \leq 2(\pi+1)\vert x-y\vert_0,
\end{equation}
and
\begin{equation}\label{egal}
F(\frac{x}{\vert x\vert_0}) - F(\frac{y}{\vert y\vert_0})=G(\frac{x}{\vert x\vert_0}) - G(\frac{y}{\vert y\vert_0}).
\end{equation}
By  \cite[Corollary 1.5]{led}, it is enough to prove that 
\begin{multline*}
       {\bf P}_{\gamma,h} \otimes  {\bf P}_{\gamma,h}  \Big[u,v\in {\bf S}_{h} : \;\; 
\left\vert F(u) - F(u)\right\vert \geq 3r\Big]\\
=    \nu_\gamma\otimes\nu_\gamma\Big[x, y\in \R^{N} : \;\; 
\Big\vert F(\frac{x}{\vert x\vert_0}) - F(\frac{y}{\vert y\vert_0})\Big\vert \geq 3r\Big]\leq C{\rm e}^{-c N r^2}.
\end{multline*}
Denote by ${\cal M}_\gamma$ the median of $x\mapsto \vert x\vert_0$ with respect to $\nu_\gamma$. Then by \eqref{egal}
\begin{multline*}
\nu_\gamma\otimes\nu_\gamma\Big[x, y\in \R^{N} : \;\; 
\Big\vert F(\frac{x}{\vert x\vert_0}) - F(\frac{y}{\vert y\vert_0})\Big\vert \geq 3r\Big] \\
\begin{aligned}
&=\nu_\gamma\otimes\nu_\gamma\Big[x, y\in \R^{N} : \;\; 
\Big\vert G(\frac{x}{\vert x\vert_0}) - G(\frac{y}{\vert y\vert_0})\Big\vert \geq 3r\Big]\\
&\leq \nu_\gamma\otimes\nu_\gamma\Big[x, y\in \R^{N} : \;\; 
\Big\vert G(\frac{x}{{\cal M}_\gamma}) - G(\frac{y}{{\cal M}_\gamma})\Big\vert \geq r\Big]
 +2\nu_\gamma\Big[x \in \R^{N}: \; \Big\vert G(\frac{x}{\vert x\vert_0}) - G(\frac{x}{{\cal M}_\gamma})\Big\vert \geq r\Big].
 \end{aligned}
\end{multline*}
By \eqref{lip} we have
$$
\Big\vert G(\frac{x}{\vert x\vert_0}) - G(\frac{x}{{\cal M}_\gamma})\Big\vert  \leq 
2(\pi+1)\left\vert\frac{\vert x\vert_0}{{\cal M}_{\gamma}} -1\right\vert,$$
which implies from \eqref{Blip} that 
$$ \nu_\gamma\Big[x \in \R^{N} : \;\; 
\Big\vert G(\frac{x}{\vert x\vert_0}) - G(\frac{x}{{\cal M}_\gamma})\Big\vert \geq r\Big]  \leq  C_{1}{\rm e}^{-c_1 {\cal M^{2}_\gamma}r^2}.$$
Similarly, by \eqref{lip} and \eqref{Blip}
$$ \nu_\gamma\Big[x,y \in \R^{N} : \;\; 
\Big\vert G(\frac{x}{{\cal M}_\gamma}) - G(\frac{y}{{\cal M}_\gamma})\Big\vert \geq r\Big]  \leq  C_{2}{\rm e}^{-c_2 {\cal M^{2}_\gamma}r^2}.$$
To conclude the proof of the proposition we use the following
\begin{lemm} In $\R^{N}$, denote by ${\cal M}_\gamma(N)$ the median of $x\mapsto \vert x\vert_0$ with respect to $\nu_\gamma$, and by  ${\cal A}_\gamma(N)$ its expectation. Then there exist $C,C_{1},C_2>0$ such that for all $N\geq 1$
$$
|   {\cal M}_\gamma(N)-{\cal A}_\gamma(N)  |\leq C\quad\text{and}\quad  C_1\sqrt N \leq {\cal M}_\gamma(N) \leq C_2\sqrt N. 
$$
\end{lemm}
\begin{proof}
Here we use the notation $|x|_{N}:=(x^{2}_{1}+\dots+x^{2}_{N})^{1/2}$. 
By definition of $\M_{\gamma}$, we have for all $t>0$ and thanks to \eqref{ExpLip}
\begin{equation*}
\frac12= {\bf P}_{\gamma,h}\Big[\, |x|_{N}\geq \M_{\gamma} \, \Big]\leq \e^{-t(\M_{\gamma}-\A_{\gamma})}{\bf E}\Big[\,\e^{ t(|x|_{N}-\A_{\gamma})}\, \Big]\leq \e^{-t(\M_{\gamma}-\A_{\gamma})}\e^{ct^{2}}.
\end{equation*}
Then, we choose $\dis t=(\M_{\gamma}-\A_{\gamma})/(2c)$ and get for some $C>0$, $|\M_{\gamma}-\A_{\gamma}|\leq C$. This was the first claim.

Next, by Cauchy-Schwarz, we obtain   $\A^{2}_{\gamma}(N)\leq \int_{\R^{N}}|x|^{2}_{N}\text{d}\nu(x)=N$. Now we prove that   there exists $C>0$ so that for all $N\geq 1$, $\A_{\gamma}(N)\geq C\sqrt{N}$. Indeed,
\begin{eqnarray*}
\A_{\gamma}(N+1)-\A_{\gamma}(N)&=&\int_{\R^{N+1}}\frac{x^{2}_{N+1}}{|x|_{N}+|x|_{N+1}}\text{d}\nu(x)\\
&\geq&\frac12 \int_{\R^{N+1}}\frac{x^{2}_{N+1}}{|x|_{N+1}}\text{d}\nu(x)=\frac{\A_{\gamma}(N)}{2(N+1)}.
\end{eqnarray*}
This implies that for all $N\geq 1$
\begin{equation*}
\A_{\gamma}(N+1)\geq (1-\frac1{2(N+1)})^{-1}\A_{\gamma}(N)\geq (1+\frac{1}{2(N+1)})\A_{\gamma}(N),
\end{equation*}
and then $\A_{\gamma}(N)\geq P_{N}\A_{\gamma}(1)$, where
\begin{equation*}
\ln P_{N}=\sum_{k=2}^{N}\ln(1+\frac1{2k})=\frac12\ln N+\mathcal{O}(1),
\end{equation*}
which yields the result.
\end{proof}


\end{document}